%% file: self_interacting_networkArXiV.tex
\documentclass[a4paper,11pt,reqno]{amsart}
\usepackage{fullpage}
\usepackage[english]{babel} 
\usepackage[utf8]{inputenc} 
\usepackage[T1]{fontenc} 
\usepackage{amsmath,amsfonts,amssymb,amscd,stmaryrd,bbm,amsthm,mathrsfs,bbm}
\usepackage[foot]{amsaddr}
\usepackage{enumerate} 
\usepackage[colorlinks=true,linkcolor=MidnightBlue,citecolor=OliveGreen]{hyperref}
\usepackage[dvipsnames]{xcolor} 

\usepackage[titletoc]{appendix}
\calclayout 

\usepackage{pgf,tikz}
\usetikzlibrary{shapes,snakes,patterns,arrows,automata,positioning,patterns}

\input{00lettres_latex.tex} 
\input{00arbres.tex} 

\input{00theorem.tex}

\newcommand{\eps}{\varepsilon}
\newcommand{\dd}{\mathop{}\!\mathrm{d}}

\newcommand{\norm}[1]{\left\lVert #1\right\rVert}
\newcommand{\abs}[1]{\left\lvert #1\right\rvert}

\newcommand{\ind}{\mathbbm{1}}	
\newcommand{\C}{C_0}
\renewcommand{\EE}[1]{\mathbbm{E}\left[#1\right]}

\let\div\relax
\DeclareMathOperator{\div}{div}

\author{R\'emi Catellier} 
\address{Unviersit\'e Côte d'Azur, CNRS, LJAD}
\email{remi.catellier@unice.fr}
\author{Yves D'Angelo} 
\address{Unviersit\'e Côte d'Azur, INRIA, CNRS, LJAD}
\email{yves.dangelo@unice.fr}

\author{Cristiano Ricci}
\address{Scuola Normale Superiore, Pisa}
\email{cristiano.ricci@sns.it}
\date{\today}
\title{A mean-field approach to self-interacting network, convergence and regularity}

\begin{document}
\begin{abstract}
The propagation of chaos property for a system of interacting particles, describing the spatial evolution of a network of interacting filaments is studied. The creation of a network of mycelium  is analyzed as representative case, and the generality of the modeling choices are discussed. Convergence of the empirical density for the particle system to its mean field limit is proved, and a result of regularity for the solution is presented. 
\end{abstract}
\maketitle
\tableofcontents

\section{Introduction}
\label{sec:introduction}
The study of complex networks has seen a growing interest in the last few years. The nature of such networks is not uniquely defined: some examples are informational networks, (of relation between individuals, citation graphs,...), technological (power grids, public transportation, computer network,...), or biological  (vascular, biochemical, neural network,...) \cite{barratDynamicalProcessesComplex2008, newmanStructureFunctionComplex2003, albertStatisticalMechanicsComplex2002, vespignaniModellingDynamicalProcesses2012}. In all the above mentioned phenomena, transformations arise from individuals: beginning by the development of a new connection between existing entities, as it often appears in neurons, or the introduction of a new individual into the system. All these contributions sum up to the evolution of the network as considers as a unit at the macroscopic level.
The mathematics of such intricate processes, which needs to be able to capture evolution at different scales, can be achieved by linking microscopic objects, which describe individuals, with their collective mean behavior. 

Let us present the model which is takled here, with its different features. We will then link it to a specific biological framework that we have in mind \cite{dikecHyphalNetworkWhole2020}.
Here we focus on the case of the spatial evolution of a complex biological network, which evolves by means of the motions of its nodes. We consider a system of second order SDEs
\begin{equation}\label{eq:sdeintro}
\begin{cases}
\dd X^{i,N}_t = V^{i,N}_t \dd t\\
\dd V^{i,N}_t = - \lambda V^{i,N}_t \dd t + \nabla C^N(t,X^{i,N}_t) \dd t + \sigma \dd B^i_t
\end{cases}t \in [T^{i,N},\Theta^{i,N}),
\end{equation}
where $(X^{i,N}_t,V^{i,N}_t)\in \RR^{d} \times \RR^{d}$ is the position of the nodes of the network on the phase space. The processes $(B^{i}_{t})_{t\geq 0 }$ are independent Brownian motions, which affect the dynamics of the nodes at the microscopic level. Here $C^N$ represent an expandable resource  for  development, taking into account the cost the network has to pay to expand towards a specific direction. 
The motion of the particles is driven by the term  $\nabla C^{N}$ and   which couples the equations  (and reflects the fact  that particles tend to move towards more available resources) and by the term $- \lambda V^{i,N}_t$ which represents the friction. 
Each  equations in \eqref{eq:sdeintro}, describing position and velocity of a node, is valid only for a limited life span $[T^{i,N},\Theta^{i,N})$, denoting the time interval where the node can contribute to the evolution of the network. Outside  this interval the particle is at rest. 

We thus consider the network skeleton made of all the trajectories of the spatial components of the nodes, i.e. 
\[
\mathcal{N}_{t} := \bigcup_{i=1}^{N^{N}_{t}} \{ X^{i,N}_{s}\,|\, s \leq t,\, s \in [T^{i,N},\Theta^{i,N})\}\subseteq \RR^{d}.
\]

\noindent We allow the number of living nodes in the network to increase  after a bifurcation event, or to decrease 
when the node ceases to exist. 
Creation and destruction of particles are provided according to a Poisson point process: modifications in the number of particles at time $t$ are affected by the configuration of the system at all times $s \leq t$. Since the number of individuals is changing in time, we introduce the total number of particles $N^N_t$ that are alive or appeared up to time $t$. We also introduce the empirical measure of the particle system \eqref{eq:sdeintro}
\[
S^{N}_{t} = \frac{1}{N}\sum_{i=1}^{N^{N}_{t}} \ind_{[T^{i,N},\Theta^{i,N})}(t)\delta_{(X^{i,N}_{t},V^{i,N}_{t})}
\]
which is a (random) finite positive measure.
We allow a branching event to appear at time $t$, in any point on the trajectory of the particles $X^{i,N}_{s}$  for $s \leq t$, with a uniform spatial distribution on the trajectory. The symbol $\delta_{\XX^N_t}$ denotes  the uniform measure on the trajectory of the particles up to time $t$:
\begin{equation}\label{eq:delta_XXintro}
\delta_{\XX^{N}_{t}}  =  \frac1N\int_{0}^{t} \sum_{i=1}^{N^N_s} \ind_{[T^{i,N},\Theta^{i,N})}(s)\big|V^{i,N}_s\big| \delta_{X^{i,N}_s}(\dd x) \dd s.
\end{equation}
Note the presence of the term $|V^{i,N}_{t}|$ in the measure $\delta_{\XX^{N}_{t}}$: the scaling by the velocity of each particle is of crucial importance in order to obtain the uniform measure on the trajectory (see also appendix \ref{appendix:curvinlinear}).
The possibility to handle the uniform measure on the particles path, as well as the coalescence events, was opened by the choice of Langevin dynamics thanks to the higher regularity of the spatial trajectory. 
In the same spirit as branching, when a moving node hits the network structure  (i.e. they superimpose) coalescence occurs: the node stops moving and a loop is created within the network.

The function $C^{N}$ is itself influenced by particles, leading to a coupled (random) PDE-SDEs system
\begin{equation}\label{eq:introCN}
\partial_t C^N = \frac{\sigma_{C}^2}{2} \Delta C^N - (K_C*\delta_{\XX^N_t} )C^{N}.
\end{equation}
The rationale behind the above equation \eqref{eq:delta_XXintro} is the following: to contribute to the development of the network, the particles need to consume some resources. This absorption mechanism affects not only the sites where the network is expanding, i.e. corresponding to the particles position, but is present along the entire trajectory.

 Having in mind a biological framework, we can imagine that the network absorbs nutrients in order to sustain itself, along its entire length. Moreover, if we want to consider the network as a solid structure, we cannot use directly the uniform measure $\delta_{\XX^{N}_{t}}$, since the trajectories of the particles are one-dimensional objects and hence are negligible for the $d$-dimensional Lebesgue measure. For this reason, we introduce the convolution kernel $K_{C}$ into equation \eqref{eq:introCN}. 
The drift $\nabla C^{N}(t,X^{i,N}_{t})$ in the particles equations encodes a very important feature of our model: the self avoidance, i.e. the fact that particles tend to avoid visiting sites which are close to their past trajectories. The values of $C^{N}$ decrease in correspondence of the network structure $\mathcal{N}_{t}$, due to the term $- (K_C*\delta_{\XX^N_t} )C^{N}$ in \eqref{eq:introCN}. Hence $\nabla C^{N}(t,x)$ is pointing towards the areas of the space which are free, i.e. those which are far from $\mathcal{N}_{t}$.

\noindent  We are interested in the mean behavior of system of equation \eqref{eq:sdeintro}-\eqref{eq:introCN} when $N$ is large, and on the propagation of chaos property:

\begin{theorem*}[Theorem \ref{theorem:main}]
As $N$ goes to infinity the couple $(S^{N},C^{N})$ converges  to the  unique measure solution of the following system of PDEs
\begin{equation}\label{eq:introPDE}
\partial_t u + v \cdot \nabla_x u - \lambda \div_v( v u ) = \frac{\sigma^2}{2} \Delta_v u - \nabla C \cdot \nabla_v u+  G(v) \overline{u} + G(v)(K* \rho) - (K*\rho) u, 
\end{equation}
\[
\partial_t \rho(t,x) = \int_{\RR^d}|v| u(t,x,v)\dd v,\quad \overline{u}(t,x) = \int_{\RR^d} u(t,x,v) \dd v,
\]
\begin{equation}\label{eq:introClimit}
\partial_t C = \frac{\sigma_C^2}{2} \Delta C - (K_C*\rho) C.
\end{equation}

\end{theorem*}
\noindent The strong coupling between all the elements,
especially the interaction with the past configuration, was the main issue when dealing with a priori estimates, see section \ref{subsect:difficulties}. We also refer to equation Subsection \ref{sec:itoformula} and Equation \eqref{eq:PDE-system} to a more detailed explanation of each terms of the previous limiting equation, and especially the meaning of each.

The literature of interacting diffusion is very extensive, starting from the earlier results \cite{mckean1967propagation, oelschlager1984martingale, oelschlager1989derivation, sznitman1991topics}.
Many works, mostly applied to population dynamics or more generally to biology, are devoted to the interplay between different species, and to the discontinuity arising from creation or destruction of individuals. The spatial component of such a discontinuity plays a major role in the progression of the system, and has been widely studied, \cite{champagnat2006unifying, champagnat2007invasion, rieumont1991, nappo1988}. 

Because of its different features, for the analysis of a network organization requires to combine all the effects of the existing connections with the evolution of the individuals: in some cases the connections can also assume a physical meaning, affecting the structural transformation. In \cite{benaim2002, benaim2011, durrett1992} the case of self-interaction is analyzed, considering the interaction of a stochastic process $(X_t)_{t\geq 0}$ with its own trajectories where $s \leq t$. More precisely, in \cite{durrett1992}, self avoidance, which is also a key feature of our work, is treated.

We are also interested in a regularity result, showing that measure solutions of \eqref{eq:introPDE} are actually regular function solutions. 
\begin{theorem*}[Theorem \ref{theorem:regularity}]
If $u$ is a measure solution of equation \eqref{eq:introPDE} $u \in C([0,T];\cM_f^+(\RR^d\times \RR^d))$  then, for all $t_0 > 0$ $u$ lies in $C([t_0,T];C^\infty_b(\RR^d\times \RR^d))$.
\end{theorem*}
Following \cite{desvillettes2001} we proved this result under some minor additional assumptions with respect to our convergence result, see Section \ref{sec:presentation} for the detailed hypothesis and proof. 

\subsection{Difficulties}\label{subsect:difficulties}
In this subsection we aim at highlighting the main difficulties we met in proving the convergence result from the discrete to the continuous model and the regularity of solutions. 
The first problem we had to solve is the convergence of the empirical measure $S^{N}_{t}$. In fact, since the total number of living particles changes over time, $S^{N}$
 a probability measure but only a finite positive measure. A proper tightness criterion in the space of finite positive measure is thus required, \cite{kipnis2013scaling}. 
In order to prove tightness an a priori bound on the total mass, i.e. on the ratio 
\[
\EE{\frac{N^{N}_{t}}{N}}
\] is needed. This is not a simple task. Since proliferation can occur with  uniform distribution at any point along the network $\mathcal{N}_{t}$, the rate of proliferation depends on the total length. 
However the rate of growth of the network, which corresponds to how much the rate of proliferation increases, is influenced by the other elements of the system. The velocity of each living particles $V^{i,N}_{t}$, which affect the expansion, is driven by the term $\nabla C^{N}$. Moreover particles velocity is also affected by the noise, hence its fluctuations may be arbitrarily large and can be controlled only in the average.  This intricacy leads to a very difficult coupled problem. 

\noindent Note that the high complexity of the system is all due to the scaling term $|V^{i,N}_{s}|$ in the uniform measure over $\mathcal{N}_{t}$ \eqref{eq:delta_XXintro}. In order to make this difficulty clearer consider the following 
\[
\widetilde{\delta_{\XX^{N}_{t}}}  =  \frac{1}{N}\int_{0}^{t} \sum_{i=1}^{N^N_s} \ind_{[T^{i,N},\Theta^{i,N})}(s) \delta_{X^{i,N}_s}(\dd x) \dd s,
\]
that is the same as $\delta_{\XX^{N}_{t}}$ where the velocity term is neglected (hence is not uniform on $\mathcal{N}_{t}$). Integrating the function $1$ and computing the expected value leads to 
\[
\EE{\left\langle \widetilde{\delta_{\XX^{N}_{t}}},1 \right\rangle }\leq \int_{0}^{t} \EE{\frac{N^{N}_{s}}{N}}\,ds.
\]
It is now clear that in this case it can be possible to obtain a closed equation for the average total mass, independently of $C^{N}$ or of the particle velocity. 
By considering $\delta_{\XX^{N}_{t}}$ this last inequality is not straightforward, since it involves the term $|V^{i,N}_{s}|$ that has to be controlled separately.

\noindent We managed to solve this issue by closing a first a priori estimate independently of the others (Lemma \ref{lemmma:gradientC^N})
\begin{equation}\label{eq:nablaCNintro}
\EE{ \sup_{t \in [0,T]}\norm{\nabla C^{N}(t,\cdot)}_{\infty}}  \leq C.
\end{equation}
Thanks to the previous estimate it is possible to obtain a control on the particle velocity, that leads to the required bound on the total mass (Lemma \ref{lemma:N^N_t/N}).\\ 

Let us now focus on the tightness of $C^{N}$. The coupling with the particles system in the equation for $C^{N}$ has the form 
\begin{equation}\label{eq:introCoupling}
(K_C*\delta_{\XX^N_t} )C^{N}.
\end{equation}
The analogous term in the equation satisfied by the empirical measure $S^{N}$ in its weak formulation, see equation \eqref{eq:itoformula}, takes the form
\[
\langle S^{N}_{s},\nabla_{v}f\cdot \nabla C^{N}\rangle,
\]
where $f$ is a test function. 
Since $S^{N}$ is converging only weakly as a probability measure we see that  uniform convergence of $\nabla C^{N}$ to $\nabla C$ is required. The bound in \eqref{eq:nablaCNintro} is not enough to prove the convergence of the first derivatives of $C^{N}$, hence we had to refine this result. Thanks to the control on the total mass previously discussed we prove  (Corollary \ref{corollary:D2C^N})
\[
\EE{\sup_{t\in[0,T]}\norm{D^{2}C^{N}(t,\cdot)}_{\infty}}\leq C.
\]
Moreover, from equation \eqref{eq:introCoupling} we also understand that it is required to prove the weak convergence in the sense of finite measure of $\delta_{\XX^{N}}$. The tightness of the sequence $\{\delta_{\XX^{N}}\}_{N\in\NN}$ is proved in the same manner of that of $\{S^{N}\}_{N\in\NN}$ (Theorem \ref{theorem:tightnessS^N}). 

Uniqueness is another difficult topic. Since we aim to prove the propagation of chaos property at the level of bounded measures, proof of  uniqueness at this level of regularity is required. We  first derive a formulation for system \eqref{eq:introPDE}-\eqref{eq:introClimit} in Fourier space, and we  understand the solution in Fourier space in its mild formulation. 
Using the technique developed in \cite{desvillettes2001} we prove some hypoelliptic estimates for the Fourier multiplier involved in the Fourier formulation. These estimates are also  used when dealing with the regularity of solutions. 
Moreover, we will us the fact that if $u$ has $(1+\beta)$ moments along the velocity component, for some $\beta >  0$ and uniformly in time, namely
\[
\sup_{t\in[0,T]} \int_{\RR^{d}}\int_{\RR^{d}}\abs{v}^{1+\beta}\,u(t,\dd x,\dd v) \leq C,
\]
then we  have a control on the Fourier transform of $u$ in the space of $(1+\beta)$-Holder continuous functions. By this last remark we will produce a Gronwall type estimate in Fourier space, proving uniqueness.

In the following, we will particularize our work to the modeling of the \emph{Podospora anserina}, a filamentous fungus which has been widely used as a model organism of research. We focused on the development of the \emph{Hyphae} of the fungus, i.e. the microscopic branching filaments which collectively form the mycelium. This is all based on a series of experiment that can be (partly) found in \cite{dikecHyphalNetworkWhole2020}.
\noindent Many works focused on the growth of a single tip of a hypha,  \cite{tabor2008}, while others studies focused on the collective evolution of the mycelium using a PDE modeling approach, \cite{boswell2012, boswell2008linking, boswell2003}. In our work we first directed our attention on a proper description of the network on the microscopic scale, taking into account the formation of new individuals (creation of new tips for the filament), the possibility of coalescence of existing branches, and linking the above mentioned phenomena with the collective behavior at the macroscopic scale. 

This paper is structured as follows: Section \ref{sec:presentation} is devoted to the rigorous presentation of the model and to all the required hypothesis in order to prove our main results. In Section \ref{sec:limitingfluid} we introduce the mathematical tools that we will need in the rest of the paper and identify (heuristically) the limiting equation of the microscopic model. We also give an accurate statement of our main results Theorem \ref{theorem:main} and Theorem \ref{theorem:regularity}. In Section \ref{sec:tightness} we isolate the a priori estimates needed for the tightness of the particle system, that we will need in order to prove our propagation of chaos results. Section \ref{sec:fourierformulation} is devoted to the Fourier formulation for the limiting equation and to a semi explicit formulation for the solution in Fourier coordinates. In Section \ref{sec:hypoelliptic} we produce some uniform in time Hypoelliptic estimates on the solution, while in Section \ref{sec:uniqueness} we prove uniqueness of measure solutions. Finally in section \ref{sec:smoothness} we prove our regularity result, applying the results developed in the previous sections. 

In all the following, the notation $a \lesssim b$ means that there exists a constant $C>0$ independent of $a$ and $b$ such that $a \le C b$, and $a \lesssim_K b$ means that the constant $C$ does depends on $K$.

\section{Rigorous description of the model}
\label{sec:presentation}

Let us now introduce the model in some detail: as stated in the introduction the case we have in mind is the growth of  a filamentous fungi. Our whole discussion stands in an arbitrary dimension $d$, even if we are more interested in  $d=2,3$.

The fundamental components of our model are the following: 
\begin{itemize}
\item \emph{Tips}: we describe the growth of a branch by the motion of its \emph{tip}, i.e. the final portion of the branch, taken in the direction of growth. The motion is described by a second order SDE, driven by a field of nutrients: position and velocity of the $i$-th particle/tip at time $t$ is denoted by $(X^{i,N}_{t},V^{i,N}_{t})$;
\item \emph{Branching and coalescing distributions}: At time $t=0$ we assume to have $N_{0} = N$ particles: however the total number of tips can change, at random,  due to \emph{branching} or \emph{coalescing} (anastomosis) events. Each tip is considered "active" for $t\in [T^{i,N},\Theta^{i,N})$ denoting its time of birth and death. Branching can appear either in the position of a tip (tip-branching), either uniformly on the set of trajectories of the particles (network-branching). These events happen accordingly to a Poisson point process as specified below;
\item \emph{Concentration of Nutrients}: Tips move according to a field of nutrients, shifting in the direction where concentration is higher. Nutrients are absorbed, in the regions of space that are occupied by the path of the tips (i.e. where the filament is present). The field of nutrients evolution is governed by a PDE for its concentration $C^N(t,x)$. This PDE is coupled with the particles SDEs.
\end{itemize}

Let us now described in some details all of the elements introduced above.

We denote by $N^{N}_{t}$ the total number of particles that are alive or have lived up to time $t$. Note that with this notation we have $N^{N}_{0} = N$ and that $N^{N}_{t}$ is an increasing process. Each of the particles satisfies a second order SDE, valid for a limited time between its birth time (which is zero if the particle is one of the $N$ ancestors) and to its death:
\[
\begin{cases}
\dd X^{i,N}_t = V^{i,N}_t \dd t\\
\dd V^{i,N}_t = - \lambda V^{i,N}_t \dd t + \nabla C^N(t,X^{i,N}_t) \dd t + \sigma \dd B^i_t
\end{cases} t \in [T^{i,N},\Theta^{i,N}),
\]
with $(X^{i,N}_{0},V^{i,N}_{0})$ according to a probability distribution function $u_{0}$ on $\RR^{d}\times \RR^{d}$, and
where $(B^{i}_{t})_{t\geq 0}$ are independent Brownian motions on $\RR^{d}$. In the previous $\lambda > 0$ represents the friction coefficient, $\sigma$ is the diffusion coefficient and $C^{N}$ is the concentration of nutrients introduced above. Here the sign in front of the gradient of the nutrient field denotes the trend for the particle  \emph{leave and avoid} the areas where the potential field is low. This  allow us to obtain the self-avoiding behavior, that is actually observed in spatial exploration phenomena \cite{dikecHyphalNetworkWhole2020}.

In order to specify the Point processes for the branching and coalescence (anastomosis), we  introduce the following empirical measures:
\[
S^N_t(\dd x, \dd v) = \frac{1}{N}\sum_{k=1}^{N^N_t} \ind_{[T^{i,N},\Theta^{i,N})}(t) \delta_{X^{i,N}_t,V^{i,N}_t}(\dd x, \dd v),
\]
\[
\bar{S}^N_t(\dd x) = \frac{1}{N}\sum_{k=1}^{N^N_t} \ind_{[T^{i,N},\Theta^{i,N})}(t) \delta_{X^{i,N}_t}(\dd x),
\]
\[
\delta_{\XX^N_t}(\dd x) = \int_0^t \frac{1}{N}\sum_{k=1}^{N^N_s} \ind_{[T^{i,N},\Theta^{i,N})}(s) |V^{i,N}_t| \delta_{X^{i,N}_s}(\dd x).
\]
The measure $\delta_{\XX^N_t}(\dd x)$ is the uniform measure on the network renormalized by the number of initial particles (see curvilinear abscissa \eqref{eq:curvilinear} in Appendix \ref{appendix:curvinlinear}). A new tip is created according to a Compound Poisson Point process $\Phi(\dd x\times \dd v \times \dd t)$   defined by its compensator
\[
G(v)\dd v \bar{S}^{N}_{t}(\dd x)\dd t + G(v) \dd v (K*\delta_{\XX^N_t})(x)\dd x \dd t 
\]
where $G(v)$ is a $pdf$ on $\RR^{d}$ and $K$ is a mollifier with compact support. The first part of the previous correspond to tip-branching while the second to network-branching. The reasoning is the following: a new tip can appear either on the spatial position of an existing particle, with a starting velocity specified by the density $G$, or it can appear, with an uniform distribution, on the network of the trajectories. However, trajectories are $1$-dimensional objects: therefore the convolution of $\delta_{\XX^N_t}(\dd x)$ with the mollifier $K$ has the purpose to account for the (nonzero) thickness of the hypha. 

Concerning the coalescence (also called anastomosis in biological context): we introduce the Compound Poisson Point process $\Psi(\dd x\times \dd v \times \dd t)$, specified by the compensator
\[
(K*\delta_{\XX^N_t})(x)S^{N}_{t}(\dd x, \dd v)\dd t.
\]
As for the network-branching the rationale behind this is the following: coalescence between a tip and an existing branch can only happen when the two superimpose. Again,
convolution with the kernel $K$ is present to take into account  the thickness of the branch. 

Finally we discuss the equation for the potential field $C^{N}$:
\[
\partial_t C^N = \frac{\sigma_{C}^2}{2} \Delta C^N - K_C*\delta_{\XX^N_t}(x) C^{N}
\]
with $C^{N}(0,x) = C^{0}(x)$ and $K_{C}$ is a mollifier with the property $\abs{\nabla K_{C}(x)}\lesssim K_{C}(x)$. The hyphae network consumes the nutrient field to stay alive and receive the proper sustain, hence the term  $- K_C*\delta_{\XX^N_t}(x) C^{N}$. 
Since the hyphae are infinitely thin  curves (as in the modeling) but do have a  transverse thickness, the kernel $K_{C}$ reintroduces that property. 
\begin{remark}
Here the choice for $K_{C}$ is really the crux of the matter: this specific hypothesis on $K_{C}$ will allow some a priori estimates to follows easily without interlacing to each other. For a discussion without this hypothesis see \cite{flandoli2017}.
\end{remark}

 We summarize below all the hypotheses for the modelling: we split our set of hypotheses into three  blocks, separating the ones needed for the tightness, uniqueness, and the additional assumptions needed for the regularity theorem \ref{theorem:regularity}. 
\begin{hypothesis}\label{hypo:PS}
(Tightness and passage to the limit, Corollary \ref{corollary:tightnessglobal}, Theorem \ref{theorem:limitsupport}):
\begin{enumerate}
\item \label{hypo:K_C} $K_C \in C^{2}_{b}$ with the property that $\abs{\nabla K_{C}(x)}\lesssim K_{C}(x)$. 
\item $K \in C^{2}_{b}$ with compact support;
\item \label{hyp:u0} $u_{0}$ probability density function on $\RR^{d}\times \RR^{d}$; 
\item \label{hyp:G}$G$ $pdf$ on $\RR^{d}$ s.t. $G(v)\dd v \sim u_{0}(\RR^{d},\dd v)$;
\item $G$ \label{hyp:moments} and $u_{0}(\RR^{d},\dd v)$ have finite $1+\overline{\beta}$ moments for some $\overline{\beta}>0$;
\item $C^{0} \in C^{2}_{b}$.

\end{enumerate}
\end{hypothesis}

\begin{remark}\label{rem:noncompact} Remark that Hypothesis \ref{hypo:PS} \eqref{hypo:K_C} prevents $K_C$ to be compactly supported. 
Indeed, up to a translation we can assume that $K_C(0)> 0$, and let $a>0$ such that $|\nabla K_C| \le a K_C$.
Let $x\in \RR^d\backslash\{0\}$ and define for $t\in[0,1]$ $g(t) = e^{a |x| t} K_C(tx)$. Note that we have $\nabla K_C(tx) \cdot x \ge - a |x| K_C(tx)$, and $g'(t) \ge 0$, which implies that $g(1) \ge g(0)$, and $K_C(x) \ge K_C(0) e^{-a |x|}$.
\end{remark}

\begin{hypothesis}\label{hypo:uniqueness}
(Uniqueness, Theorem \ref{theorem:uniqueness})
\begin{itemize}
\item The function $C^{0} \in H^{m+2}$ for some $m> \frac{d}{2}$;
\item The function $u_{0} \in H^{m}$ for some $m > \frac{d}{2}$;
\item The Kernel $K_{C}\in C_{b}^{m+2}\cap H^{2m+2}$ for some $m>\frac{d}{2}$.
\end{itemize}
\end{hypothesis}    

\begin{hypothesis}\label{hypo:smoothness}
(Smoothness, Theorem \ref{theorem:regularity}): 
Denoting by $\hat{f}$ the Fourier transform of $f$
\begin{itemize}
\item For all $N\geq 0$, 
$\sup_{\xi\in\RR^d} |\hat{G}(\xi)\big|\big(1+|\xi|^2)^N < \infty$.
\item For all $N\geq 0$, 
$\sup_{\xi\in\RR^d} |\hat{K}(k)\big|\big(1+|\xi|^2)^N < \infty$.
\item
For all $m \geq 0$, $C_0,K_C \in H^m$ and $K_C \in C^m_b$.
\end{itemize}
\end{hypothesis}


\section{Limit Fluid   equations and precise results}\label{sec:limitingfluid}
\subsection{Notation and function spaces}
From now on we  denote by $\mathcal{M}^{+}_{f}(\RR^{d}\times \RR^{d})$, the space of all the finite positive measures over $\RR^{d}\times \RR^{d}$ and $\C(\RR^d\times\RR^d)$ the space of continuous functions vanishing at infinity. By the Riesz-Markov theorem, $\C$ is the dual of $\cM_f^+$ for the weak convergence of measure. Given $(\phi_{k})_{k\in\NN}$ a countable dense subset of $\C(\RR^{d}\times \RR^{d})$ we define
\[
\delta(\mu,\nu)= \sum_{k=1}^{\infty}\frac{1}{2^{k}}\frac{\abs{ \left\langle \mu,\phi_{k} \right\rangle - \left\langle \nu,\phi_{k} \right\rangle}}{1+\abs{ \left\langle \mu,\phi_{k} \right\rangle - \left\langle \nu,\phi_{k} \right\rangle}}  
\]
which makes $\mathcal{M}^{+}_{f}(\RR^{d}\times \RR^{d})$ into a complete metric space, which topology corresponds to weak
convergence of measure. 
We  then denote by $\mathcal{D}([0,T];\mathcal{M}^{+}_{f}(\RR^{d}\times \RR^{d}))$ the space of all the càdlàg functions from $[0,T]$ to $\mathcal{M}^{+}_{f}(\RR^{d}\times \RR^{d})$, endowed with the Skorohod topology. 

We also denote by $C^{1}(\RR^{d})$  the space of all $C^{1}(\RR^{d})$ functions, endowed with the topology of uniform convergence over compact sets. We recall  this topology is generated by the following metric 
\[
d(f,g) = \sum_{N=1}^{\infty} 2^{-N} \norm{f-g}_{C^{1}\big(B(0,N)\big)} \wedge 1.
\]
We shall also introduce the spaces
\[
X = \mathcal{D}\Big([0,T];\mathcal{M}^{+}_{f}(\RR^{d}\times \RR^{d})\Big) \times \mathcal{D}\Big([0,T];\mathcal{M}^{+}_{f}(\RR^{d})\Big) \times C\Big([0,T];C^{1}(\RR^{d})\Big)
\]
endowed with the product metric introduced above, and 
\[
X_{b} = C\Big([0,T];\mathcal{M}^{+}_{f}(\RR^{d}\times \RR^{d})\Big) \times C\Big([0,T];\mathcal{M}^{+}_{f}(\RR^{d})\Big) \times C\Big([0,T];W^{1,\infty}(\RR^{d})\Big)
\]
again with the product metric. 

\subsection{It\^o formula and limit equation}
\label{sec:itoformula}
Let us remind that we are interesting by the convergence of the following empirical measures :
\begin{equation}\label{eq:def_SN}
S^N_t(\dd x, \dd v) = \frac1N\sum_{k=1}^{N^N_t} \ind_{[T^{i,N},\Theta^{i,N})}(t) \delta_{X^{i,N}_t,V^{i,N}_t}(\dd x, \dd v),
\end{equation}
\begin{equation}\label{eq:def_SNbar}
\bar{S}^N_t(\dd x) = \frac1N\sum_{k=1}^{N^N_t} \ind_{[T^{i,N},\Theta^{i,N})}(t) \delta_{X^{i,N}_t}(\dd x),
\end{equation}
\begin{equation}\label{eq:def_RN}
\delta_{\XX^N_t}(\dd x) = \int_0^t \frac1N\sum_{k=1}^{N^N_s} \ind_{[T^{i,N},\Theta^{i,N})}(s) |V^{i,N}_s| \delta_{X^{i,N}_s}(\dd x) \dd s,
\end{equation}
for $\beta \leq \overline{\beta}$
\begin{equation}\label{eq:def_deltabeta}
\delta_{\XX^N_t}^{\beta}(\dd x) = \int_0^t \frac1N\sum_{k=1}^{N^N_s} \ind_{[T^{i,N},\Theta^{i,N})}(s) |V^{i,N}_s|^{1+\beta} \delta_{X^{i,N}_s}(\dd x) \dd s
\end{equation}
and by the convergence of the function $C^N$.

For every $C^2$ function $f : \RR^{2d} \to \RR$, we finally have, by It\^o formula, used here for the $X^{i,N},V^{i,N}$ and using the fact that branching and merging are encoded thanks to Poisson Point Processes,
\begin{align}\label{eq:itoformula}
\langle S^{N}_{t},f\rangle - \langle u_0,f\rangle = & 
    \underbrace{\int_{0}^{t}\langle S^{N}_{s},v\cdot \nabla_{x}f\rangle \dd s}_{\text{kinetic equation}} 
    -\underbrace{\int_{0}^{t}\langle S^{N}_{s},\lambda v \cdot \nabla_{v}f\rangle \dd s}_{\text{friction}}\\ 
    &+
    \underbrace{\int_{0}^{t}\langle S^{N}_{s},\nabla_{v}f\cdot \nabla C^{N}\rangle \dd s}_{\text{potential}}
    + \underbrace{\frac{\sigma^2}{2} \int_{0}^{t}\langle S^{N}_{s},\Delta_{v}f\rangle \dd s}_{\text{noise on the velocity}}\nonumber\\
    &+\underbrace{\int_{0}^{t}\int_{\RR^{2d}} f(x,v) G(v) \dd v \overline{S}^{N}_{s}(\dd x)\dd s }_{\text{creation at the new tips}}\nonumber\\
    &+\underbrace{\int_{0}^{t}\int_{\RR^{2d}} f(x,v) G(v) (K*\delta_{\XX^{N}_s})(x)\dd x \dd v\dd s}_{\text{creation on the network}}\nonumber\\
    &-\underbrace{\int_{0}^{t} \int_{\RR^{2d}} f(x,v) (K*\delta_{\XX^{N}_s})(x)S^{N}_{s}(\dd x,\dd v)\dd s }_{\text{coalescence/anastomosis}}\nonumber\\
    &+ \underbrace{M^{1,N,f}_{t}}_{\substack{\text{martingale remainder}\nonumber\\ \text{of motion}}} 
    + \underbrace{M^{2,N,f}_{t}}_{\substack{\text{martingale remainder}\nonumber\\ \text{of branching}}} + \underbrace{M^{3,N,f}_{t}}_{\substack{\text{martingale remainder}\nonumber\\ \text{of anastomosis}}}.
\end{align}
The explicit martingale terms are the following :
\[
M^{1,N,f}_{t} = \int_{0}^{t}\frac{\sigma}{N}\sum_{i=1}^{N^{N}(s)} \ind_{[T^{i,N},\Theta^{i,N})}(s) \nabla_{v}f(X^{i,N}_{s},V^{i,N}_{s}) \cdot \dd B^{i}_{s}
\]
\[
M^{2,N,f}_{t}\hspace{-0.1cm}= \hspace{-0.1cm}\int_{0}^{t}\hspace{-0.1cm}\int_{\RR^{2d}} \hspace{-0.1cm}f(x,v)\Big[\Phi(\dd x\times \dd v\times \dd s) - G(v)\overline{S}^{N}_{s}(\dd x)\dd v \dd s- G(v)(K*\delta_{\XX^{N}_{s}})(x)\dd x\dd v \dd s\Big]
\]
\[
M^{3,N,f}_{t}=-\int_{0}^{t}\int_{\RR^{2d}} f(x,y)\Big[\Psi(\dd x\times \dd v\times \dd s) - (K*\delta_{\XX^{N}_{s}})(x) S_{s}^{N}(\dd x,\dd v) \dd s\Big]
\]
and
\[\partial_t C^N = \frac{\sigma^2_C}{2} \Delta C^N - K_C*\delta_{\XX^N} C^N.\]

This gives us the wanted limit equations:
\begin{equation}\label{eq:PDE-system}
\left\{
\begin{aligned}
\partial_t u + v \cdot \nabla_x u - \lambda \div_v( v u ) =& \frac{\sigma^2}{2} \Delta_v u - \nabla C \cdot \nabla_v u \\ 
&+  G(v) \overline{u} + G(v)(K* \rho) - (K*\rho) u \\
\partial_t \rho(t,x) =& \int_{\RR^d}|v| u(t,x,v)\dd v \\
\overline{u}(t,x) =& \int_{\RR^d} u(t,x,v) \dd v \\
\partial_t C =& \frac{\sigma_C^2}{2} \Delta C - (K_C*\rho) C\\
u(0,\cdot,\cdot) = u_0,\quad C&(0,\cdot) = C_0,\quad \rho(0,\cdot)= 0.
\end{aligned}
\right.
\end{equation}

Note that in the previous system, each wanted behavior of the modeling can be retrieved. Indeed, the LHS of the first equation remind us that we are working in the phase space. The first term of the RHS of the first equation comes from the Brownian motions, whereas the second term emphases that the dynamic is directed by the nutrient field, and the last three terms encode the branching-coalescence process (first the branching on the network, then the branching at the tips and finally the coalescence on the network). The second equation is the equation for the skeleton of the network. The third equation is the equation for the density of tips, and the fourth equation describe how the nutrient is used by the network. 



%

\subsection{Definitions of measure solutions and main Theorem}
\begin{definition}\label{def:solutionsystemPDE}
A measure solution of system of equation  \eqref{eq:PDE-system}  is a triple $(u,\rho,C) \in X_b$, for every test function $\phi \in C^{\infty}_{c}(\RR^{d}\times \RR^{d})$ one has
\begin{multline*}
\int_{\RR^{d}\times\RR^{d}} \phi(x,v)u(t,\dd x,\dd v) - \int_{\RR^{d}\times\RR^{d}}\phi(x,v) u(0,\dd x, \dd v)  \\
-\int_{0}^{t}\int_{\RR^{d}\times\RR^{d}}v \cdot \nabla_{x}\phi(x,v)u(s,\dd x,\dd v)\dd s + \lambda\int_{0}^{t}\int_{\RR^{d}\times\RR^{d}}\nabla_{v}\phi(x,v)\cdot v u(s,\dd x, \dd v)\dd s\\
= \frac{\sigma^{2}}{2}\int_{0}^{t}\int_{\RR^{d}\times\RR^{d}}\hspace{-0.3cm}\Delta_{v}\phi(x,v)u(s,\dd x,\dd v)\dd s + \int_{0}^{t}\int_{\RR^{d}\times\RR^{d}}\hspace{-0.3cm}\nabla_{v}\phi(x,v)\cdot \nabla C(s,x)u(s,\dd x,\dd v)\dd s\\
+\int_{0}^{t}\int_{\RR^{d}\times\RR^{d}}    \phi(x,v)G(v)\overline{u}_{s}(\dd x)\dd v\dd s+ \int_{0}^{t}\int_{\RR^{d}\times\RR^{d}}\phi(x,v)G(v)(K*\rho(s,\cdot))(x)\dd x\dd v\dd s\\
- \int_{0}^{t}\int_{\RR^{d}\times\RR^{d}}\phi(x,v)(K*\rho(s,\cdot))(x)u(s,\dd x,\dd v)\dd s, 
\end{multline*}

\[
\overline{u}(t,\dd x) = u(t,\dd x, \RR^{d}), 
\]
\noindent the measure $\rho$ satisfies, for every test function $\phi \in C^{\infty}_{c}(\RR^{d})$
\[
\int_{\RR^{d}}\phi(x)\rho(t,\dd x) = \int_{\RR^{d}} \phi(x)\rho(0,\dd x) + \int_{0}^{t} \int_{\RR^{d}} \phi(x)\abs{v} u(s,\dd x,\dd v)\dd s,
\]

\noindent and the function $C$ satisfies
\begin{multline}
\int_{\RR^{d}}C(t,x)\phi(x)\dd x = \int_{\RR^{d}}C(0,x)\phi(x)\dd x \\
+\int_{0}^{t}\int_{\RR^{d}}\frac{\sigma_{c}^{2}}{2}C(s,x)\Delta\phi(x)\dd x \dd s- \int_{0}^{t}\int_{\RR^{d}}C(s,x)(K_{C}*\rho(t,\cdot))(x)\phi(x)\dd x\dd s.
\end{multline}

\end{definition}
We end up this section by properly stating our main results:
\begin{theorem}[Propagation of chaos]\label{theorem:main}
If the hypothesis \ref{hypo:PS} are satisfied, the family of laws $\{Q^N\}_{N \in \NN}$  of the triple $(S^{N},\delta_{\XX^{N}},C^{N})_{N \in \NN}$ is tight on the space $X$. Moreover $\{Q^N\}_{N \in \NN}$ converges weakly in $X$ to $\delta_{(u,\rho,C)}$, where the triple $(u,\rho,C) \in X_b$ 
 is the unique measure solution of system of equation $\eqref{eq:PDE-system}$.
\end{theorem}

\noindent We prove the tightness of the sequence in Section \ref{sec:tightness}. The proof concludes itself in Section \ref{sec:uniqueness}  by proving uniqueness for the limit system. 

\begin{theorem}[Smoothness of solutions]\label{theorem:regularity}
Under Hypothesis \ref{hypo:PS}, 
\ref{hypo:uniqueness} and 
 \ref{hypo:smoothness}, for all $t_0>0$, the solution $u$ constructed in Theorem \ref{theorem:main} lies in $C([t_{0},T];C^{\infty}_{b}(\RR^{d}\times\RR^{d}))$. 
\end{theorem}

\noindent The proof of this result can be found in Section \ref{sec:smoothness}.

\begin{remark}
The proof of Theorem \ref{theorem:regularity} gives us a bit better, actually denoting by $\hat{u}$ the Fourier transform of $u$,  $\hat{u}(t,\cdot,\cdot)$ is decaying faster than any polynomials, impliying for example the $ \hat{u}(t,\cdot,\cdot)\in H^m$ for all $m\in \NN$.
\end{remark}

\section{A priori estimates on microscopic model and tightness}\label{sec:tightness}
We start by some a priori estimates on the microscopic model, needed for the tightness, and thus the propagation of chaos.
\subsection{A priori estimates}
\begin{lemma}
\label{lemma:randomFeynmanKac}
Suppose that $w(t,x)$ is of class $C^{1,2}([0,T]\times \RR^{d})$ and satisfies the following PDE:
\begin{equation}\label{eq:PDErandomFeynmanKac}
\left\{
\begin{aligned}
&\partial_{t}w(t,x) = \frac{a^{2}}{2}\Delta w(t,x) - \big(K_C*f(t,\cdot)\big)(x) w(t,x)\quad (t,x) \in [0,T]\times \RR^{d}\\
&w(0,x) = w_{0}(x)
\end{aligned}
\right.
\end{equation}
where $f\in L^{\infty}([0,T];\mathcal{M}^{+}_{f}(\RR^{d}))$, $K_C \in C^{2}_{b}(\RR^{d})$ is non negative and satisfies $\abs{\nabla K}\lesssim K$. Assume also that $w_{0}\in C^{2}_{b}(\RR^{d})$. Then 
\begin{equation}\label{eq:randomFeynmanKac1}
\sup_{t\in [0,T]}\norm{w(t,\cdot)}_{C^{1}_{b}} \leq C\norm{w_{0}}_{C^{1}_{b}}
\end{equation}
and
\begin{equation}\label{eq:randomFeynmanKac2}
\sup_{t\in [0,T]}\norm{w(t,\cdot)}_{C^{2}_{b}}\lesssim_{T} (1+\norm{w_{0}}_{C^{2}_{b}})\norm{K}_{C^{2}_{b}}\sup_{t\in[0,T]}f(t,\RR^{d})
\end{equation}
\end{lemma}
\begin{proof}
By Feynman-Kac representation we have 
\begin{equation}\label{eq:randomFeynmanKacBase}
w(t,x) = \EE{ \exp{\left( -\int_{0}^{t}(K_C*f(s,\cdot))(x+aW_{t}-aW_{s})\dd s \right)}w_{0}(x+aW_{t})},
\end{equation}
where $(W_{t})_{t \geq 0}$ is a standard Brownian motion.
Notice that, for every $g \in C^{2}_{b}(\RR^{d})$ with $\abs{\nabla g} \lesssim g$ one has
\[
\nabla e^{-g} = -\nabla g e^{-g}
\]
so that, using the hypothesis on the first derivative we have 
\[
\abs{\nabla e^{-g}} \leq 1.
\]
Moreover for each $i, j$
\[ \abs{\partial_{i}\partial_{j}e^{-g}}\leq \abs{\partial_{i}g}+\abs{\partial_{i}\partial_{j}g} \lesssim \norm{g}_{C^{2}_{b}}.\]
Applying the previous computation in equation \eqref{eq:randomFeynmanKacBase} we get to \eqref{eq:randomFeynmanKac1}. 
Finally
\[ \norm{D^{2}w(t,\cdot)}_{\infty}\leq \norm{w_{0}}_{C^{2}_{b}}\norm{K_C}_{C^{2}_{b}}\sup_{t\in[0,T]}f(t,\RR^{d}) \] 
which leads to \eqref{eq:randomFeynmanKac2}.

\begin{remark}\label{remark:randomFeynmanKac}
Note that, thanks to the assumptions, the bound \eqref{eq:randomFeynmanKac1} of the previous lemma is independent of $K_C*f$. This uniform bound is in fact the main reason for hypothesis \eqref{hypo:K_C} in \ref{hypo:PS}.
\end{remark}
\begin{lemma}
\label{lemmma:gradientC^N}
Under assumptions \ref{hypo:PS} we have
\begin{equation}\label{eq:aprioriGradientC^N}
\sup_{t\in[0,T]}\norm{ C^{N}(t,\cdot)}_{C^{1}_{b}} \leq C\norm{C^{0}}_{C^{1}_{b}} \quad a.s.
\end{equation}
\end{lemma}
\begin{proof}
By the hypothesis \ref{hypo:PS} we have
$K*\delta_{\XX^{N}} \in C([0,T];C^{\infty}(\RR^{d}))$ and $\abs{\nabla (K*\delta_{\XX^{N}})}$ $= \abs{(\nabla K)*\delta_{\XX^{N}}} \lesssim K*\delta_{\XX^{N}}$ almost surely. We can then apply the same strategy as Lemma \ref{lemma:randomFeynmanKac} and obtain the desired result.
\end{proof}
\begin{lemma}\label{lemma:estimatesOnV}
There exist a sequence of i.i.d. random variables $\chi_{i}$ satisfying $\EE{\exp{(a\chi_{i}^{2})}}<\infty$ for some $a > 0$, such that, for all $\beta \leq \overline{\beta}$ (where $\bar{\beta}$ is defined in Hypothesis  \ref{hypo:PS}),
\begin{equation}\label{eq:estimatesOnV1}
\sup_{t\in[0,T]} \abs{V^{i,N}_{t}}^{1+\beta}\lesssim_{T} Z^{i,\beta} := \abs{V^{i,N}_{T^{i,N}}}^{1+\beta} + \chi_{i}^{1+\beta} + \norm{C^{0}}_{C^{1}_{b}}^{1+\beta}
\end{equation}
and
\begin{equation}\label{eq:estimatesOnV2}
\abs{X^{i,N}_{t}-X^{i,N}_{s}} + \abs{V^{i,N}_{t}-V^{i,N}_{s}}\lesssim Z^{i,0}\abs{t-s}^{\frac{1}{2}-\eps}
\end{equation}
for all $0 <\eps<\frac{1}{2}$.
Moreover for all $\beta$ the variables $(Z^{i,\beta})_{i\in\NN}$ are i.i.d r.v. 
\end{lemma}
\begin{proof}
For $t\in[T^{i,N},\Theta^{i,N})$
\begin{equation}\label{eq:explicitV^i}
V^{i,N}_{t} = V^{i,N}_{T^{i,N}}e^{-\lambda(t-T^{i,N})}+ \int_{T^{i,N}}^{t} e^{-\lambda(t-r)}\nabla C^{N}(r,X^{i,N}_{r})\dd r+\sigma\int_{T^{i,N}}^{t}e^{-\lambda(t-r)}\dd B^{i}_{r}.
\end{equation}
Call
\[ U^{i}_{t,t'} =  \int_{t}^{t'}e^{-\lambda(t'-r)}\dd B^{i}_{r}\]
and notice that these are Gaussian random variables. Moreover there is a constant $b_{\lambda}$ depending on $\lambda$ such that 
\[ \textrm{Var}(\lvert{U^{i}_{t,t'}-U^{i}_{s,s'}}\vert) \lesssim b_{\lambda} \abs{t-s}+\abs{t'-s'}\]
and, since the variables are Gaussian, we have
\[ \EE{ \left( \frac{\lvert{U^{i}_{t,t'}-U^{i}_{s,s'}}\vert}{\abs{t-s}+\abs{t'-s'}}\right)^{2k}} \lesssim  \frac{b_{\lambda}^{k}(2k)!}{2^{k}k!}.\]
Hence, by the Kolmogorov regularity Theorem and Garsia Rodemich Rumsey inequality (see \cite{garsiaRealVariableLemma1970}), there exists a random variable $\chi_{i}$, $\EE{ \exp{(a\chi_{i}^{2})}}<\infty$ for some $a>0$ depending on $\lambda$, such that 
\begin{equation}\label{eq:regularityU}
\vert U^{i}_{t,t'}-U^{i}_{s,s'} \vert \lesssim \chi_{i} (\abs{t-s}+\abs{t'-s'})^{\frac{1}{2}-\eps}
\end{equation}
a.s. for every $\eps \in (0,\frac{1}{2})$. Note that the variables $\chi_{i}$ depends only on $B^{i}_{t}$ and $\lambda$ and thus are i.i.d. Plugging the last inequality into \eqref{eq:explicitV^i}, and using \eqref{eq:aprioriGradientC^N} we obtain the first part of the lemma. The fact that the variables $Z^{i,\beta}$ are independent is a consequence of  the independence of $\chi_{i}$ and of $V^{i,N}_{T^{i,N}}$.
\end{proof}
\begin{remark}\label{remark:estimatesOnV}
Notice that for all $i\in \NN$ the random variable $Z^{i,\beta}$ is independent on $T^{i,N}$: in fact $Z^{i,\beta}$ depends only on the variables $\chi_{i}$, which are independent on $T^{i,N}$, and on $V^{i,N}_{T^{i,N}}$ which is distributed as $G(v)\dd v$. This fact will be crucial in the proof of Lemma \ref{lemma:N^N_t/N}.
\end{remark}

\begin{lemma}\label{lemma:N^N_t/N}
With the same notation of Lemma, \ref{lemma:estimatesOnV} for all $\beta \leq \overline{\beta} $
\begin{equation}\label{eq:sumVi}
\EE{ \frac{1}{N}\sum_{i=1}^{N^{N}_{t}}\abs{V^{i,N}_{t}}^{1+\beta}}\lesssim_{T}\EE{ \frac{1}{N}\sum_{i=1}^{N^{N}_{t}}Z^{i,\beta}}=\EE{Z^{1,\beta}}\EE{\frac{N^{N}_{t}}{N}}. 
\end{equation}

Moreover it holds
\begin{equation}\label{eq:N^N_t/N}
\EE{\frac{N^{N}_{T}}{N} } \lesssim_{T} 1.
\end{equation}
\end{lemma}   
\begin{proof}

\end{proof}
Let us first start by proving the first part of the lemma. We can control 
\[ \EE{ \sum_{i=1}^{N^{N}_{t}}\abs{V^{i,N}_{t}}^{1+\beta}}\lesssim_{T} \EE{ \sum_{i=1}^{N^{N}_{t}}Z^{i,\beta}}\] 
using inequality \eqref{eq:estimatesOnV1} in Lemma \ref{lemma:estimatesOnV}. Recall that $Z^{i,\beta}$ is independent on $T^{i,N}$: we can now prove a variant of Wald identity. 
\[ \EE{ \sum_{i=1}^{N^{N}_{t}}Z^{i,\beta}}\hspace{-0.1cm} = \hspace{-0.1cm}\sum_{k=1}^{\infty} \sum_{i=1}^{k}\EE{Z^{i,\beta}\ind_{N^{N}_{t}=k} } \hspace{-0.1cm}= \hspace{-0.1cm}\sum_{i=1}^{\infty}\sum_{k=i}^{\infty}\EE{Z^{i,\beta}\ind_{N^{N}_{t}=k} } = \sum_{i=1}^{\infty}\EE{Z^{i,\beta}\ind_{N^{N}_{t}>i} }\]   
\[ = \sum_{i=1}^{\infty}\EE{Z^{i,\beta}\ind_{T^{i,N}\leq t}} = \sum_{i=1}^{\infty}\EE{Z^{i,\beta}}\EE{\ind_{T^{i,N}\leq t} } = \EE{Z^{i,\beta}}\EE{N^{N}_{t}}.
\]  
This proves the first part. Concerning the second, 
 by applying It\^o formula \eqref{eq:itoformula} with $f\equiv 1$ we get
\begin{multline*}\label{eq:itoformulaf1}
\frac{N^{N}_{t}}{N} = \frac{N^{N}_{0}}{N} + \int_{0}^{t}\frac{N^{N}_{s}}{N}\dd s + \int_{0}^{t}\int_{0}^{s}\frac{1}{N}\sum_{i=1}^{N^{N}_{r}}\ind_{[T^{i,N},\Theta^{i,N}]}\abs{V^{i,N}_{r}}\dd r\dd s\\ -\int_{0}^{t}\int_{\RR^{d}}(K*\delta_{\XX^{N}_{s}})(x)S^{N}_{s}(\dd x,\dd v)\dd s + M^{2,N,1}_{t} + M^{3,N,1}_{t}.
\end{multline*}
Taking the expected value and neglecting the indicator function, we obtain \eqref{eq:N^N_t/N} by  applying the generalized Grönwall lemma of Appendix \ref{appendix:gronwall}. 

\end{proof}
\begin{remark}
This part of the proof follows the same strategy as \cite{flandoli2017} regarding the variant of Wald identity. The r.v. $Z^{i,\beta}$ are obtained by a different type argument, using the a priori bound on $\nabla C^{N}$ available in our case. 
\end{remark}

\begin{corollary}\label{corollary:D2C^N} 
\[ \EE{ \sup_{t\in[0,T]}\norm{D^{2}C^{N}(t,\cdot)}_{\infty}} \lesssim_{T}\norm{C^{0}}_{C^{2}_{b}}\norm{K}_{C^{2}_{b}} \EE{\sup_{t\in[0,T]}\frac{N^{N}_{t}}{N}}\]
\end{corollary}
\begin{proof}
By Lemma \ref{lemma:randomFeynmanKac} it's enough to verify that $\delta_{\XX^{N}}(\cdot,\dd x)\in L^{\infty}([0,T];\mathcal{M}^{+}_{f}(\RR^{d}))$ almost surely. Note that $\delta_{\XX^{N}}(\cdot,\RR^{d})$ is increasing in time, so that it's enough to bound $\delta_{\XX^{N}}(T,\RR^{d})$. Using Lemma \ref{lemma:N^N_t/N} we have
\[
\EE{\delta_{\XX^{N}}(T,\RR^{d})} \leq \int_{0}^{T}\EE{\frac{1}{N}\sum_{i=1}^{N^{N}_{s}}\abs{V^{i,N}_{s}}}\dd s \lesssim_{T} \EE{\frac{N^{N}_{T}}{N}},
\]
which ends the proof thanks to Lemma \ref{lemma:N^N_t/N}
\end{proof}
\subsection{Tightness of the laws and passage to the limit}
In order to prove the tightness, we rely on standard argument (see for example \cite{coppolettaCriterionConvergenceMeasure1986}). Nevertheless in order to be as exhaustive as possible, and since the convergence takes place in $\cM_f$ (instead of the more standard space of probability measure), we give an extended proof of it.  
\begin{theorem}\label{theorem:tightnessS^N}
The sequence $\{Q_{S}^{N}\}_{N\in\NN}$ of the laws of the empirical measure $\{S^{N}_{\cdot}\}_{N\in\NN}$ is tight on $\mathcal{D}([0,T];\mathcal{M}^{+}_{f}(\RR^{d}\times \RR^{d}))$. Moreover, the sequence $\{Q_{\delta_{\XX}^{\beta}}^{N}\}_{N\in\NN}$ of the laws of the empirical measure $\{\delta_{\XX^{N}_{\cdot}}^{\beta}\}_{N\in\NN}$, defined in \eqref{eq:def_deltabeta}, is tight on $\mathcal{D}([0,T];\mathcal{M}^{+}_{f}(\RR^{d}))$ for every $\beta \leq \overline{\beta}$.
\end{theorem}
\begin{proof}
In order to prove the tightness of the laws of $Q_{S}^{N}$ 
on $\mathcal{D}([0,T];\mathcal{M}^{+}_{f}(\RR^{d}\times \RR^{d}))$ 
we will show that $Q_{S}^{N}\circ \Phi_{k}$ is tight on $\mathcal{D}([0,T];\RR)$, for every function $\Phi_{k}$ in a dense subfamily of $\C(\RR^{d}\times \RR^{d})$ (recall the definition of $M^{+}_{f}$ in the introduction). In particular the functions $\Phi_{k}$ can be taken to be Lipschitz continuous. Thanks to Aldous criterion (\cite{kipnis2013scaling}), to prove the tightness of $Q_{S}^{N}\circ \Phi_{k}$ it enough to verify the following conditions:
\begin{equation}\label{eq:aldus1}
\forall t\in[0,T], \forall \eps > 0, \exists R \textrm{ s.t. } \sup_{N\in\NN} Q_{S}^{N} \left( \abs{ \left\langle \pi_{t},\Phi_{k} \right\rangle} > R \right) < \eps
\end{equation}
and
\begin{equation}\label{eq:aldus2}
\forall \delta, \lim_{\gamma \to 0}\limsup_{N\in\NN}\sup_{\substack{\tau \in \mathcal{L}_{T} \\ \theta \leq \gamma}} Q^{N}_{S} \left( \abs{ \left\langle \pi_{\tau+\theta},\Phi_{k} \right\rangle - \left\langle \pi_{\tau},\Phi_{k} \right\rangle} > \delta \right) = 0,
\end{equation}
where $\mathcal{L}_{T}$ is the family of all stopping times bounded by $T$.
Notice that 
\[
 Q_{S}^{N} \left( \abs{ \left\langle \pi_{t},\Phi_{k} \right\rangle} > R \right) = \PP \left( \abs{\left\langle S^{N}_{t},\Phi_{k}  \right\rangle } > R \right) \leq \frac{1}{R}\EE{ \abs{ \left\langle S^{N}_{t},\Phi_{k} \right\rangle } } \leq \frac{1}{R}\EE{\frac{N^{N}_{t}}{N}}\norm{\Phi_{k}}_{\infty}
\]
so that condition \eqref{eq:aldus1} follows by Lemma \ref{lemma:N^N_t/N}. Concerning condition \eqref{eq:aldus2} we have
\[
Q^{N}_{S} \left( \abs{ \left\langle \pi_{\tau+\theta},\Phi_{k} \right\rangle - \left\langle \pi_{\tau},\Phi_{k} \right\rangle} > \delta \right) = \PP \left( \abs{ \left\langle S^{N}_{\tau+\theta},\Phi_{k}\right\rangle- \left\langle S^{N}_{\tau},\Phi_{k} \right\rangle } > \delta \right)
\]
furthermore
\[
\abs{ \left\langle S^{N}_{\tau+\theta},\Phi_{k}\right\rangle- \left\langle S^{N}_{\tau},\Phi_{k} \right\rangle } \leq \frac{1}{N}\sum_{i=N^{N}_{\tau}+1}^{N^{N}_{\tau+\theta}}\abs{\Phi_{k}(X^{i,N}_{\tau+\theta},V^{i,N}_{\tau+\theta})-\Phi_{k}(X^{i,N}_{\tau},V^{i,N}_{\tau}) }
\]
\[
\leq \norm{\Phi_{k}}_{Lip}\frac{1}{N}\sum_{i=N^{N}_{\tau}+1}^{N^{N}_{\tau+\theta}}\left( \abs{X^{i,N}_{\tau+\theta}-X^{i,N}_{\tau}}+\abs{V^{i,N}_{\tau+\theta}-V^{i,N}_{\tau}} \right) 
\]
\[
\leq \norm{\Phi_{k}}_{Lip}\frac{1}{N}\sum_{i=N^{N}_{\tau}+1}^{N^{N}_{\tau+\theta}}Z^{i,0}\,\theta^{\frac{1}{2}-\eps}
\]
by Lemma \ref{lemma:estimatesOnV}. 
By Lemma \ref{lemma:N^N_t/N} we have
\[
\EE{\frac{1}{N}\sum_{i=N^{N}_{\tau}+1}^{N^{N}_{\tau+\theta}}Z^{i,0} } \leq \EE{ \frac{1}{N} \sum_{i=1}^{N^{N}_{T}}Z^{i,0} } = \EE{\frac{N^{N}_{T}}{N}}\EE{ Z^{1,0} }.
\]
and thus, by using Markov inequality
\[
\PP \left( \abs{ \left\langle S^{N}_{\tau+\theta},\Phi_{k}\right\rangle- \left\langle S^{N}_{\tau},\Phi_{k} \right\rangle } > \delta \right)\leq \norm{\Phi_{k}}_{Lip}\frac{C}{\delta}\theta^{\frac{1}{2}-\eps}.
\]
Letting $\gamma \to 0$ ends the first part of the proof. For the second part, to prove the tightness of $Q_{\delta_{\XX}^{\beta}}^{N}$, we apply the same strategy: Thus we need to proove that for every Lipschitz function $\Phi_{k}$, in dense subset of $\C(\RR^{d})$, we have 
\begin{equation}\label{eq:aldus3}
\forall t\in[0,T], \forall \eps > 0 \exists R \textrm{ s.t. } \sup_{N\in\NN} Q_{\delta_{\XX}^{\beta}}^{N} \left( \abs{ \left\langle \pi_{t},\Phi_{k} \right\rangle} > R \right) < \eps
\end{equation}
and
\begin{equation}\label{eq:aldus4}
\forall \delta, \lim_{\gamma \to 0}\limsup_{N\in\NN}\sup_{\substack{\tau \in \mathcal{L}_{T} \\ \theta \leq \gamma}} Q_{\delta_{\XX}^{\beta}}^{N} \left( \abs{ \left\langle \pi_{\tau+\theta},\Phi_{k} \right\rangle - \left\langle \pi_{\tau},\Phi_{k} \right\rangle} > \delta \right) = 0
\end{equation}
For the first condition we have
\[
Q_{\delta_{\XX}^{\beta}}^{N} \left( \abs{ \left\langle \pi_{t},\Phi_{k} \right\rangle} > R \right) = \PP \left( \abs{ \left\langle \delta_{\XX^{N}_{t}}^{\beta},\Phi_{k} \right\rangle } > R \right)
\]
and
\begin{align*}
\EE{ \abs{ \left\langle \delta_{\XX^{N}_{t}}^{\beta},\Phi_{k} \right\rangle }} 
= & 
\EE{\int_{0}^{t}\frac{1}{N}\sum_{i=1}^{N^{N}_{s}}\abs{V^{i,N}_{s}}^{1+\beta} \Phi_{k}(X^{i,N}_{s})\dd s}
\\
 \leq & 
 \norm{\Phi_{k}}_{\infty}\int_{0}^{t}\EE{ \frac{1}{N}\sum_{i=1}^{N^{N}_{s}}\abs{V^{i,N}_{s}}^{1+\beta} } \dd s 
 \\ 
 \leq & \norm{\Phi_{k}}_{\infty}\EE{ Z^{i,\beta} } \int_{0}^{t} \EE{ \frac{N^{N}_{s}}{N} } \dd s
    \\ \lesssim_T & \norm{\Phi_{k}}_{\infty}\EE{ Z^{1,\beta} } 
\end{align*}
so that we can obtain condition \eqref{eq:aldus3} by Markov inequality and Lemma \ref{lemma:N^N_t/N}. 
Furthermore
\[
\EE{\abs{ \left\langle \delta_{\XX^{N}_{t}}^{\beta}(\tau+\theta),\Phi_{k} \right\rangle  - \left\langle \delta_{\XX^{N}_{t}}^{\beta}(\tau),\Phi_{k}  \right\rangle  }} \leq \norm{\Phi_{k}}_{\infty}\int_{\tau}^{\tau+\theta} \hspace{-0.2cm}
\EE{\frac{1}{N}\sum_{i=N^{N}_{\tau}+1}^{N^{N}_{\tau+\theta}}\abs{V^{i,N}_{s}}^{1+\beta}}\dd s.
\]
Again, using Lemma \ref{lemma:N^N_t/N} 
\[
\EE{\frac{1}{N}\sum_{i=N^{N}_{\tau}+1}^{N^{N}_{\tau+\theta}}\abs{V^{i,N}_{s}}^{1+\beta}}\leq \EE{\frac{1}{N}\sum_{i=1}^{N^{N}_{T}} Z^{i,\beta}} = \EE{\frac{N^{N}_{T}}{N}} \EE{ Z^{1,\beta} }.
\]
leading to 
\[
\EE{\abs{ \left\langle \delta_{\XX^{N}_{t}}^{\beta}(\tau+\theta),\Phi_{k} \right\rangle  - \left\langle \delta_{\XX^{N}_{t}}^{\beta}(\tau),\Phi_{k}  \right\rangle  }} \leq \norm{\Phi_{k}}_{\infty}\EE{\frac{N^{N}_{T}}{N}}\EE{ Z^{1,\beta}} \theta
\]
This ends the proof of condition \eqref{eq:aldus4} and thus the proof of the lemma. 
\end{proof}

\begin{lemma}\label{lemma:WCN}
For every $\eps > 0 $ there exists $R$ such that
\begin{equation}\label{eq:WCN}
\PP \Big( \norm{C^{N}}_{W^{1,\infty}\big([0,T];C^{0}(\RR^{d})\big)} > R \Big) < \eps
\end{equation}
\end{lemma}
\begin{proof}
Let us first notice that    
\begin{multline*}
    \PP \Big( \norm{C^{N}}_{W^{1,\infty}\big([0,T];C^{0}(\RR^{d})\big)} > R \Big) \leq \PP \Big( \sup_{t \in [0,T]}\norm{\partial_{t}C^{N}(t,\cdot)}_{\infty} > \frac{R}{2} \Big)\\ + \PP \Big( \sup_{t \in [0,T]}\norm{C^{N}(t,\cdot)}_{\infty} > \frac{R}{2} \Big)
\end{multline*}
and that the second term can be made arbitrary small by Markov inequality and Lemma \ref{lemmma:gradientC^N}.

By a direct computation we also have
\begin{multline*}
\norm{\partial_{t}C^{N}(t,\cdot)}_{\infty} \lesssim_{T} \frac{\sigma_{C}^{2}}{2}\norm{ D^{2}C^{N}(t,\cdot)}_{\infty} 
\\
+ \left(\frac{1}{N}\int_0^T \sum_{i=1}^{N^N_s} \ind_{[T^{i,N},\Theta^{i,N})}(s) |V^{i,N}_s| \dd s\right) \norm{K_{C}}_{\infty} \norm{C^{N}(t,\cdot)}_{\infty}.
\end{multline*}
Taking the supremum in time we obtain
\begin{multline*}
\sup_{t\in[0,T]}\norm{\partial_{t}C^{N}(t,\cdot)}_{\infty} 
\lesssim_{T} 
\frac{\sigma_{C}^{2}}{2}\sup_{t\in[0,T]}\norm{ D^{2}C^{N}(t,\cdot)}_{\infty} 
\\ + 
	\left(
		\frac{1}{N}
		\int_0^T 
			\sum_{i=1}^{N^N_s} 
				\ind_{[T^{i,N},\Theta^{i,N})}(s) 
				|V^{i,N}_s| 
		\dd s
	\right)\norm{K_{C}}_{\infty} \sup_{t\in[0,T]}\norm{C^{N}(t,\cdot)}_{\infty}
\end{multline*}
thus
\begin{align*}
\PP \Big( \sup_{t \in [0,T]}\norm{\partial_{t}C^{N}(t,\cdot)}_{\infty} > \frac{R}{2} \Big) 
\lesssim_{T}& 
\frac{2\sigma^{2}_{C}}{R}\EE{ \sup_{t\in[0,T]}\norm{ D^{2}C^{N}(t,\cdot)}_{\infty}}
\\
&+\frac{2\norm{K_{C}}_{\infty}}{\sqrt{R}}\EE{ \left(\frac{1}{N}\int_0^T \sum_{i=1}^{N^N_s} \ind_{[T^{i,N},\Theta^{i,N})}(s) |V^{i,N}_s| \dd s\right)} 
\\
&+ \frac{2}{\sqrt{R}}\EE{ \sup_{t\in[0,T]}\norm{C^{N}(t,\cdot)}_{\infty} }.
\end{align*}
By taking $R$ sufficiently large we can conclude by Lemmas \ref{lemmma:gradientC^N}, \ref{lemma:N^N_t/N}  and Corollary \ref{corollary:D2C^N}.
\end{proof} 

\begin{theorem}\label{theorem:tightnessCCompact}
Denoting by $B_{M}$ the $d$-dimensional ball of radius $M$ we have that $\forall M \in \NN$ the sequence $\{Q_{C}^{N,M}\}_{N\in\NN}$ of the laws of the function $\{C^{N}_{\cdot}\}_{N\in\NN}$  restricted to $B_{M}$, is tight on $C\Big([0,T];C^{1}(B_{M})\Big)$.

\end{theorem}
\begin{proof}
By Simon's lemma we have that 

\[
W^{1,\infty}\Big([0,T];C(B_{M})\Big)\cap L^{\infty}\Big([0,T];C^{2}(B_{M})\Big)
\]
is compactly embedded into 
\[
 C\Big([0,T];C^{1}(B_{M})\Big).
\]
Thus, the set 
\[
K_{M,R,S} := \left\{ f\,\bigg\vert\, \norm{f}_{W^{1,\infty}\Big([0,T];C^{0}(B_{M})\Big)} \leq R, \norm{f}_{L^{\infty}\Big([0,T];C^{2}(B_{M})\Big)} \leq S\right\} 
\]
is a compact subset of $ C\Big([0,T];C^{1}(B_{M})\Big)$ with respect to the strong topology. We have
\begin{multline*}
Q_{C}^{N,M} \left( K_{M,R,S}^{c} \right) \leq \PP \left( \norm{C^{N}_{|_{B_{M}}}}_{W^{1,\infty}\Big([0,T];C^{0}(B_{M})\Big)}>R \right)\\ + \PP \left( \norm{C^{N}_{|_{B_{M}}}}_{L^{\infty}\Big([0,T];C^{2}(B_{M})\Big)}> S  \right)
\end{multline*}
\begin{multline}\label{eq:tightnessCCompact}
\leq \PP \left( \norm{C^{N}}_{W^{1,\infty}\Big([0,T];C^{0}(\RR^{d})\Big)}>R \right) + \frac{1}{S}\EE{ \sup_{t\in[0,T]}\norm{D^{2}C^{N}(t,\cdot)}_{C(\RR^{d})} }.\\
\end{multline}
By Lemma \ref{lemma:WCN} and Corollary \ref{corollary:D2C^N}, choosing $R$ and $S$ big enough, we can make \eqref{eq:tightnessCCompact} arbitrary small, ending the proof. 
\end{proof}
From the previous lemma we immediately get the following theorem:

\begin{theorem}\label{theorem:tightnessCR}
The sequence $\{Q_{C}^{N}\}_{N\in\NN}$ of the laws of the function $\{C^{N}_{\cdot}\}_{N\in\NN}$ is tight on $C\Big([0,T];C^{1}(\RR^{d})\Big)$. 
\end{theorem}
\begin{proof}
The result follows by the tightness of the sequence of the laws $\{Q^{N,M}_{C}\}_{N\in\NN}$.
\end{proof}

With a little abuse of notation we will refer to the laws of $\delta_{\XX^{N}}$ as $Q_{\delta_{\XX}}^{N} = Q_{\delta_{\XX}^{0}}^{N}$. We will also denote by $Q^{N}$ the measure
\[
Q^{N} = Q^{N}_{S}\otimes Q_{\delta_{\XX}}^{N} \otimes Q^{N}_{C}
\]
defined on the product space $X$.

By Theorems \ref{theorem:tightnessS^N} and \ref{theorem:tightnessCR} we immediately get the following corollary:
\begin{corollary}\label{corollary:tightnessglobal}
The sequence $\{Q^{N}\}_{N\in\NN}$ of probability measure is tight on the space $X$. 
\end{corollary}

\begin{theorem}\label{theorem:limitsupport}
Any limit point of any subsequence of the sequence $\{Q^{N}\}_{N\in\NN}$, is supported on the measure solutions of system of equation \eqref{eq:PDE-system}.
\end{theorem}
\begin{proof}[Sketch of the proof]
The fact that limit objects satisfy system of equations  \eqref{eq:PDE-system} is classical, see \cite{kipnis2013scaling}.
Hence we highlight only the main difficulties. Let us show that all the reminders in the It\^o formulations, the martingales $M^{k,N,f}_{t}$ for $k = 1,2,3$, vanish when $N$ tends to infinity.  For every test function $\phi \in C^{\infty}_{b}(\RR^{d}\times \RR^{d})$ we have to check that (recall It\^o formula in section \ref{sec:itoformula}) 
\begin{equation}
\EE{ \sup_{t \in [0,T]} \abs{M^{k,N,f}_{t}}^{2}}\stackrel{N \rightarrow \infty}{\longrightarrow} 0
\end{equation}
for $k=1,2,3$. By using Burkholder inequality we will conclude  by showing that the quadratic variation of all these martingale goes to zero in $L^{1}(\Omega)$. The first martingale, coming from the Brownian motion is classical. Concerning the martingale $M^{2,N,f}$, deriving from the branching process, we first note that we can rewrite
\[
M^{2,N,f}_{t} = \frac{1}{N}\sum_{i=1}^{N^{N}_{t}} M^{i,2,N,f}_{t}
\]
where $M^{i,2,N,f}_{t}$ are martingales defined as
the integral of $f$ with respect to a compensated Poisson point process $\Phi^{i}$, whose compensator is the random measure
\[
G(v)\ind_{s \in [T^{i,N},\theta^{i,N})}\delta_{X^{i,N}_{s}}(\dd x)\dd v\dd s + \int_{0}^{s} G(v)\ind_{s \in [T^{i,N},\theta^{i,N})}\abs{V^{i,N}_{r}}\delta_{X^{i,N}_{r}}(\dd x)\dd v\dd r \dd s.
\]
It follows that 
\begin{multline*}
\left[ M^{2,N,f}_{\cdot} \right]_{T} = \frac{1}{N^{2}} \sum_{i,j=1}^{N^{N}_{T}}\left[ M^{i,2,N,f}_{\cdot},M^{j,2,N,f}_{\cdot} \right]_{T} \\= \frac{1}{N}\int_{0}^{T} \int_{\RR^{d}} f(x,v)^{2}\Big[G(v)\overline{S}^{N}_{s}(\dd x)\dd v \dd s - G(v)\delta_{\XX^{N}_{s}}(\dd x)\dd v\dd s\Big]\\ + \frac{1}{N^{2}} \sum_{i\neq j}\left[ M^{i,2,N,f}_{\cdot},M^{j,2,N,f}_{\cdot} \right]_{T}
\end{multline*}

It is now clear that if  the terms corresponding to $i \neq j$ vanishes, we will obtain the desired result. To do so, observe that for $i \neq j$, shortening the notation to $M^{i}_{t}$ and $M^{j}_{t}$
\[
\left[ M^{i}_{\cdot},M^{j}_{\cdot} \right]_{T} = [(M^{i}_{\cdot}),(M^{j}_{\cdot})]^{c}_{T} + \Delta \left[ M^{i}_{\cdot},M^{j}_{\cdot} \right]_{T} = [(M^{i}_{\cdot})^{c},(M^{j}_{\cdot})^{c}]_{T} + \Delta M^{i}_{T} \Delta M^{j}_{T}
\]
where $(M^{i})^{c}$ denotes the continuous part of the variation. The continuous part obviously vanish, being the motion of particles for $i \neq j$ driven by independent Brownian motions. For the jump part, notice that the probability of two birth or coalescence events, corresponding to jumps, to happen at the same time is zero: this is a consequence of the conditional independence with respect to $\mathcal{F}_{t}$ of $M^{2,N,f}$ and $M^{3,N,f}$, as well as that of the martingales $M^{i}$ and $M^{j}$. The proof for the martingale $M^{3,N,f}$ follows in the same manner. 

Concerning the time regularity of the limit points we just remark the fact that limits point are probability measures on $X_b$. They are continuous in time as  consequence of the tightness criterion in the space $\mathcal{D}$ see \cite{kipnis2013scaling}. 
\end{proof}



\begin{remark}\label{remark:1+beta}
In Theorem \ref{theorem:tightnessS^N} we have seen that the for every $\beta \leq \overline{\beta}$ the sequence of the laws of $\delta_{\XX^{N}}^{\beta}$ is tight. Starting from this fact, we consider the following equation 
\begin{equation}\label{eq:pde1+beta}
\left\{
\begin{aligned}
\partial_t \rho^{\beta}(t,x) =& \int_{\RR^d}|v|^{1+\beta} u(t,x,v)\dd v \\
\rho^{\beta}(0,x) =& \,\,0.\\
\end{aligned}
\right. 
\end{equation}
We can then prove with small modification of Theorem \ref{theorem:limitsupport} that the sequence of laws of $\delta_{\XX^{N}}^{\beta}$ is supported on the weak solutions of the previous PDE, and thus that the whole sequence of laws 
$\{Q_{\delta_{\XX}^{\beta}}^{N}\}_{N\in\NN}$ is converging in $\mathcal{D}([0,T];\mathcal{M}^{+}_{f}(\RR^{d}))$ to the unique solution of \eqref{eq:pde1+beta}. Hence we deduce the following fact:
\end{remark}
\begin{corollary}\label{corollary:pde1+beta}
If the function $u_{0}$ starting condition of the system \eqref{eq:PDE-system} satisfying  hypothesis \eqref{hyp:u0},\eqref{hyp:G}, satisfies also hypothesis \eqref{hyp:moments}, then the solution of system \eqref{eq:PDE-system} has the property
\[
\sup_{t \in [0,T]} \rho^{\beta}(t,\RR^{d}) = \int_{\RR^d}|v|^{1+\beta} u(t,\RR^{d},\dd v) < \infty
\]
for all $\beta \leq \overline{\beta}$.
\end{corollary}

\section{Fourier (re)formulation of the PDE}\label{sec:fourierformulation}

\subsection{The coupled PDEs in Fourier space}
In the previous section, we proved existence of measure solution of the system of Equations \eqref{eq:PDE-system}, when $u_0,\rho_0$ are bounded measures with moments of sufficiently high order and $C_0$ is regular enough :
\begin{equation}
\left\{ 
\begin{aligned}\label{eq:pdf}
        \partial_t u(t,x,v) + &v\cdot\nabla_x u(t,x,v) - \lambda \div_v(v u(t,x,v)\big)\\ =& \frac{\sigma^2}{2}\Delta_v u(t,x,v)  + \nabla C(t,x).\nabla_v u(t,x,v)\\
        &\quad  + G(v)\bar{u}(t,x) + G(v)K*\rho(t,x) 
          - u(t,x,v)K*\rho(t,x)\\
        \rho(t,x) =& \int_0^t \int_{\RR^d} |v| u(s,x,v) \dd v \dd s\\
        \bar{u}(t,x) = &\int_{\RR^d} u(t,x,v) \dd v\\
        \partial_t C(t,x) &= \frac{\sigma_C^2}{2} \Delta C(t,x) -K_C*\rho(t,x) C(t,x).
\end{aligned}
\right.
\end{equation}

The aim of this section is to prove uniqueness of such solutions. We will use Fourier techniques  from \cite{desvillettes2001} and hypoelliptic estimates. The proof of uniqueness will be decomposed in several steps. The first one is an a priori bound on $C$ in Sobolev spaces. The second one will be a reminder of the derivation of the formulation of the system \eqref{eq:pdf} in Fourier space, as long as a mild formulation in the Fourier space. The third step will be the proof of hypoelliptic estimates (following \cite{desvillettes2001}) for the Fourier multiplier involved in the Fourier expression of \eqref{eq:pdf}. Finally, by using a Grönwall type argument, we will conclude about the uniqueness of measure solutions. The hypoelliptic bound we also be useful in the last Subsection \ref{sec:smoothness} to prove smoothness of the solutions. 


\subsection{A priori bound for $C$}

Let us suppose that $(u,\rho,C)$ is a measure solution of Equation \eqref{eq:PDE-system}, as in Definition \ref{def:solutionsystemPDE}. To prove uniqueness, we will strongly use the smoothness of the potential field $C$. This lemma is obviously an extension of Lemma \ref{lemma:randomFeynmanKac}, we state it and prove it for the sake of the comprehension. 
\begin{lemma}\label{lem:smoothness_C}
Let $m\in \NN$, $f \in L^\infty\big([0,T];\cM^+_f(\RR^d)\big)$, $g\in L^\infty\big([0,T];H^m(\RR^d,\RR)\big)$, $K\in C^m_b(\RR^d;\RR_+)$ and $w_0\in H^m(\RR^d;\RR_+)$. 
let $w$ be the solution of the heat equation 
\[\partial_t w = \frac{a^2}2 \Delta w - (K*f) w + g,\quad w(0,\cdot) = w_0.  \]
Then $w\in L^\infty\big([0,T];H^m(\RR^d;\RR)\big)$, and one has for all $t\in[0,T]$,
\[\|w(t,\cdot)\|_{H^m} \lesssim \big(1+\|K\|_{C^m_b}\big)^m \big(1+\sup_{s\in [0,T]}f(s,\RR^d)\big)^m \bigg(\|w_0\|_{H^m}+\int_0^t\|g(s,\cdot)\|_{H^m} \dd s \bigg)\]
\end{lemma}

\begin{proof}
Let $W$ be a standard Brownian motion on $[0,T]$, as in Lemma \ref{lemma:randomFeynmanKac},  we have the explicit formula
\begin{multline*}
w(t,x) = \mathbbm{E}\Big[e^{-\int_0^t K*f\big(r,x+a(W_t-W_r)\big)\dd r} w_0(x+ aW_t)\\ + \int_0^t g\big(s,x+a(W_t-W_s)\big)e^{-\int_s^t K*f(r,x+aW_t-aW_r) \dd r} \dd s.\Big].
\end{multline*} 
Hence, thanks to Fa\'a di Bruno formula, for all multi-index $\alpha$ with $|\alpha|\le m$, we have 
\begin{multline*}
\partial_\alpha w(t,x) = 
\sum_{|\beta|+|\gamma|\le |\alpha|}c_{\alpha,\beta,\gamma} \mathbbm{E}\bigg[ \partial_{\beta} w_0(x+ a W_t) \partial_\gamma \Big(e^{-\int_0^t K*f\big(r,\cdot + a (W_t-W_r) \big)\dd r}\Big)(x) \\ + \int_0^t \partial_\beta g\big(s,x+a(W_t-W_s)\big) \partial_\gamma \Big(e^{-\int_s^t K*f(r,\cdot+aW_t-aW_r) \dd r}\Big)(x) \dd s \bigg]
\end{multline*}
for certain constants $c_{\alpha,\beta,\gamma}>0$.
Remark also that 
\[\bigg|\partial_\gamma \Big(e^{-\int_0^t K*f\big(r,\cdot + a (W_t-W_r) \big)\dd r}\Big)(x) \bigg| \lesssim \big(1+\|K\|_{C^m_b}\big)^m \big(1+\sup_{s\in[0,T]} f(s,\RR^d)\big)^m.\]
Hence, 
\begin{multline*}
\int_{\RR^d} \big|\partial_\alpha w(t,x) \big|^2  \dd x \lesssim \big(1+\|K\|_{C^m_b}\big)^m \big(1+\sup_{s\in[0,T]} f(s,\RR^d)\big)^m 
\\
\sum_{|\beta|+|\gamma|\le |\alpha|} \mathbbm{E}\bigg[ \int_{\RR^d} \big|\partial_\beta w_0(x+aW_t)\big|^2\dd x + \int_0^t \int_{\RR^d} \Big| \partial_\beta g\big(s,x+a(W_t-W_s)\big)\Big|^2 \dd x \dd s \bigg]
\end{multline*}
which gives, by summing over $\alpha$, 
\[\|w(t,\cdot)\|_{H^m} \lesssim \big(1+\|K\|_{C^m_b}\big)^m \big(1+\sup_{s\in[0,T]} f(s,\RR^d)\big)^m  \bigg( \|w_0\|_{H^m} + \int_0^t \|g(s,\cdot)\|_{H^m} \dd s\bigg),\]
which is the wanted result.
\end{proof}

\subsection{The equation in Fourier space}\label{subsec:fourier}

Note that when $u$ (and its associated $\rho$, see Definition \ref{def:solutionsystemPDE}) are finite positive measure solutions, their associated Fourier transforms (respectively in space and velocity) exist as bounded functions. Furthermore we have
\[\hat{u}(t,k,\xi) = \int_{\RR^d} \int_{\RR^d} e^{-i k \cdot x} e^{-i\xi \cdot v} u(t,\dd x, \dd v),\]
\[\hat{\bar u}(t,k) = \int_{\RR^d\times \RR^d}  e^{-ik\cdot x} u(t,\dd x, \dd v) = \hat u(t,k,0).\]
and 
\[\hat{\rho}(t,k) = \int_0^t \int_{\RR^d} \int_{\RR^d} e^{-i k \cdot x}  |v| u(t,\dd x, \dd v) \dd s.\]
Hence, as Schwartz distributions,
\[\hat{\rho}(t,k) = \int_0^t (-\Delta)^{\frac12} \hat{u}(s,k,0)\dd s = \int_0^t \frac{\Gamma\Big(\frac{d+1}{2}\Big)}{\pi^{\frac{d+1}{2}}}\, P.V. \int_{\RR^d} \frac{\hat u(s,k,0) - \hat u(s,k,\xi)}{|\xi|^{d+1}} \dd \xi \dd s,\]
where $P.V.$ denotes the principal value. By using Lemma \ref{lemma:FractionalLaplacian} of the Appendix, a way of controlling the sup norm of $\hat \rho(t,k)$ is to control the $(1+\beta)$-Hölder norm of $\xi\to\hat u(t,k,\xi)$, uniformly in $t,k$.
Finally, in Fourier space (as Schawrtz distributions), Equation \eqref{eq:pdf} becomes

\begin{equation}    
\left\{\begin{aligned}\label{eq:pdf_fourier}
        \partial_t \hat u(t,k,\xi) &- k\cdot\nabla_\xi \hat{u}(t,k,\xi) + \lambda \xi\cdot\nabla_\xi \hat{u}(t,k,\xi)\\ &= -|\xi|^2\frac{\sigma^2}{2}\hat{u}(t,k,\xi)  +i \xi \cdot \widehat{\nabla C} * \hat u(t,k,\xi)\\
        &\quad  + \hat{G}(\xi)\big(\hat{u}(t,k,0) +\hat{K}  \hat \rho(t,k)\big)  -  \Big( \hat{K} \hat{\rho}\Big) * \hat u(t,k,\xi).\\
        \hat{\rho}(t,k) &= \int_0^t(-\Delta)^{\frac12}\hat u(s,k,0)\dd s\\
        \partial_t C(t,x) &= \frac{\sigma_C^2}{2}\Delta C(t,x)-\cF^{-1}\Big(\hat{K}_C \hat{\rho}\Big)(t,x)C(t,x).
\end{aligned}
\right.
\end{equation}
We know that there exist a solution of the previous equation, seen as Schwarz distribution $(\hat u, \hat \rho,C)$. We can rewrite the previous system in its Mild formulation (see Appendix \ref{appendix:Mild}).
For $k,\xi\in \RR^d$ let us define $\xi(t) = \left(\xi - \frac{k}{\lambda}\right)e^{-\lambda t} + \frac{k}{\lambda}$, and one have the following semi-explicit formula for the first equation of \eqref{eq:pdf_fourier} :
\begin{equation}\label{eq:MildFourier}
\begin{aligned}
\hat u(t,k,\xi) = & \hat u_0\big(k,\xi(t)\big)\exp\bigg(-\frac{\sigma^2}{2}\int_0^t |\xi(r)|^2 \dd r \bigg) \\
&+\int_0^t  
        i\xi(t-s)\cdot\big(\widehat{\nabla C}*\hat{u}\big)\big(s,k,\xi(t-s)\big)e^{-\frac{\sigma^2}{2}\int_0^{t-s} |\xi(r)|^2 \dd r 
    }\dd s \\&
        + \int_0^t G\big(\xi(t-s)\big)\big(\hat{u}(s,k,0) + \hat{K} \hat{\rho}(s,k)\big) e^{-\frac{\sigma^2}{2}\int_0^{t-s} |\xi(r)|^2 \dd r 
    } \dd s \\&
      - \int_0^t \Big(\big( \hat{K} \hat{\rho}\big)*\hat{u}\Big)\big(s,k,\xi(t-s)\big)
      e^{-\frac{\sigma^2}{2}\int_0^{t-s} |\xi(r)|^2 \dd r 
    }\dd s.
\end{aligned}   
    \end{equation}

\section{Hypoelliptic estimates}\label{sec:hypoelliptic}

\begin{proposition}\label{prop:hypo}
There exists a universal constant $c>0$ such that for all $t\ge 0$ and all $k,\xi\in \RR^d$, we have
\begin{equation}\label{eq:schauder}
e^{-\frac{\sigma^2}{2}\int_0^t|\xi(r)|^2 \dd r} \le e^{-c\frac{\sigma^2}{2}\left(\int_0^t \left(\frac{1-e^{-\lambda r}}{\lambda}\right)^2 \dd r |k|^2 + \int_0^t e^{-2\lambda r} \dd r|\xi|^2\right)}
\end{equation}
and for all $n\ge 0$,
\[|\xi(t)| \bigg(\int_0^t \big|\xi(r)\big| \dd r\bigg)^n e^{-\frac{\sigma^2}{2}\int_0^t|\xi(r)|^2 \dd r} \lesssim c \frac{t^{\frac{n-1}{2}}}{\sigma^n}.\]
\end{proposition}

\begin{proof}
Let us remark that 
\[
\begin{aligned}
\int_0^t|\xi(r)|^2 \dd r =& |\xi|^2\int_0^t e^{-2 \lambda r}  \dd r + |k|^2 \int_0^t \left(\frac{1-e^{-\lambda r}}{\lambda}\right)^2\dd r  \\&\quad+ 2 \xi\cdot k \int_0^t e^{- \lambda r}\left(\frac{1-e^{-\lambda r}}{\lambda}\right)\dd r\\
=&\big(A(t) \Xi(t)\big)\cdot  \Xi(t) 
\end{aligned}
\]
where 
\[\Xi(t) = \begin{pmatrix}   \left(\int_0^t e^{-2\lambda r} \dd r\right)^\frac{1}{2} \xi \\ \left(\int_0^t \left(\frac{1-e^{-\lambda r}}{\lambda}\right)^2 \dd r\right)^\frac12 k \end{pmatrix}\]
and
\[A(t) = \begin{pmatrix} Id & a(t) I_d \\ a(t) Id & Id  \end{pmatrix},\] 
\[a(t) = \frac
{
\int_0^t 
        e^{-\lambda r}
        \left(
            \frac{1-e^{-\lambda r}}{\lambda}
        \right) 
\dd r
}
{
\left(
    \int_0^t 
        \left(
            \frac{1-e^{-\lambda r}}{\lambda}
        \right)^2
    \dd r
    \int_0^t 
        e^{-2\lambda r} 
    \dd r
\right)^\frac{1}{2}
}.\]
The matrix $A(t)$ has two eigenvalues : $1+a(t)$ and $1-a(t)$, hence
\[\big(A(t) \Xi(t)\big)\cdot  \Xi(t) \ge \big(1-a(t)\big)|\Xi(t)|^2.\]
Furthermore, by taking 
\[g(x)= \frac{\int_x^1 1-u\dd u}{\left(\int_x^1 u \dd u \int_x^1 \frac{(1-u)^2}{u} \dd u\right)^{1/2}} = \frac{(1-x)^{\frac{3}{2}}}{\Big((1+x)\big(-2 \log(x) - (1-x)(3-x)\big)\Big)^{1/2}}, \]
one has
$a(t) =g(e^{-\lambda t})$.
Furthermore $g(0)=0$, $g(x) \to_{x\to 1} \frac{\sqrt3}{2}$, and 
\[g'(x) =\frac{2(x+2)(1-x)^\frac12h(x)}{\Big((1+x)\big(-2\log(x)-(1-x)(3-x)\big)\Big)^{3/2}},\]
with $h(x) = \frac{(1-x)(1+5x)}{2 x (x+2)} + \log(x)$.
Finally
$h'(x) = -\frac{(1-x)^3}{x^2(x+2)^2}$, and $h(1)=0$, hence $h\ge 0$ and so does $g'(x)$. Hence $g$ is non-decreasing  and
$a(t) = g(e^{-\lambda t}) \le \frac{\sqrt{3}}{2}$, and the smallest egeinvalue of the matrix is greater than $1-\frac{\sqrt3}{2}$. Hence one have
\begin{multline*}
e^{-\frac{\sigma^2}{2} \int_0^t |\xi(r)|^2 \dd r} \le e^{-\frac{(1-a(t))\sigma^2}{2}  |\Xi(t)|^2 }\\ \le \exp\Bigg(-\frac{ 1-\sqrt{3}}{2}\frac{\sigma^2}{2} \bigg(\int_0^t e^{-\lambda 2r}\dd r |\xi|^2 +\int_0^t \big(\frac{1-e^{-\lambda r}}{\lambda}\big)\dd r |k|^2 \bigg)\Bigg)
\end{multline*}
 and \eqref{eq:schauder} holds true with $c= 1-\frac{\sqrt3}{2}$.
 
For the second inequality, notice that 
\[|\xi(t)| \le |\xi| e^{-\lambda t} + |k| \frac{1-e^{-\lambda t}}{\lambda}.\]
Furthermore, $\sqrt{\int_0^t e^{-2\lambda r} \dd r} \ge \sqrt{t} e^{-\lambda t}$ and by convexity,
\[\sqrt{\int_0^t \Big(\frac{1-e^{-\lambda r}}{\lambda}\Big)^2 \dd r} \ge  \frac{1-e^{-\lambda t}}{\lambda}\sqrt{\int_0^t \frac{r^2}{t^2} \dd r} = \sqrt{t}\frac{1-e^{-\lambda t}}{\sqrt{3}\lambda}.\]
Hence, there exists a constant such that 
\[|\xi(t)| \lesssim \frac{1}{\sigma\sqrt{t}} \Big(\sigma^2|\Xi(t)|^2\Big)^{1/2}.\]
Using the fact that $|\Xi|$ is increasing in time, we have by using Equation \eqref{eq:schauder}
\begin{align*}\big|\xi(t)\big| \bigg(\int_0^t  |\xi(s)|\dd s\bigg)^n e^{-\frac{\sigma^2}{2}\int_0^t| \xi(s)|^2\dd s} \\ \lesssim & \frac{1}{\sigma^{n+1} \sqrt{t}} \sigma^{n+1} |\Xi(t)|^{n+1} \bigg(\int_0^t \frac{1}{\sqrt{s}}\dd s\bigg)^n e^{-\frac{c \sigma^2}{2}|\Xi(t)|^2} \end{align*}
which allows us to conclude easily.
\end{proof}

\section{Uniqueness of solutions}\label{sec:uniqueness}
In order to prove uniqueness, we will use a Gronwall type argument.

\begin{theorem}\label{theorem:uniqueness}
Let $G$, $K_C$, $K$, $C_0$ and $u_0$ which satisfies the Hypothesis \ref{hypo:PS} and \ref{hypo:uniqueness}.
 Then there exists a unique measure solution for equation \eqref{eq:pdf}.
\end{theorem}

The proof is decomposed into three parts. We first obtain a priori bounds for $C$, $\rho$ and $u$ for any measure solution. We then show how to control the difference of the $\rho$ and the $C$ parts of two solutions by the difference of the $u$ parts. Finally we control the difference of the $u$ part and use a Gronwall type argument to have uniqueness. 

\begin{proof}

\emph{A priori bounds for the solutions:}

Note first that when  $(u,\rho,C)$ is a measure solution, since $\rho$ is a  finite measure and $C_0 \in H^{m+2}$ (thanks to Lemma \ref{lem:smoothness_C})
 , $C\in C\big([0,T];H^{m+2}(\RR^d)\big)$ (take $g=0$ and $f=\rho$) with a bound which depends only on $\sup_{t\in[0,T]}\rho(t,\RR^d)$ and of $\|K_C\|_{C^{m+2}_b}$, and since $m>\frac{d}2$,
\[\|\widehat{\nabla C}(t,\cdot)\|_{L^1} \lesssim \|C(t,\cdot)\|_{H^{m+2}} \lesssim_{T,C_0,K_C}  \big(1+\sup_{s\in[0,T]} \rho(s,\RR^d)\big)^{m+2}.\]
Finally, $\rho(s,\RR^d) = \int_0^s \int_{\RR^d\times \RR^d} |v| u(r,\dd x,\dd v) \dd r$ and $u$ is a non negative measure for all time, hence $\rho$ is non-decreasing in time and $\sup_{s\in[0,T]} \rho(s,\RR^d) = \rho(T,\RR^d)$.
Thanks to Young inequality, one have
\[\|\widehat{\nabla C}*\hat{\rho}(s,\cdot)\|_{L^\infty} \lesssim \big(1+\rho(T,\RR^d)\big)^{m+2} \|\hat\rho(s,\cdot)\|_{L^\infty}.\] 
Furthermore, thanks to the hypothesis and Corollary \ref{corollary:pde1+beta}, $u$ has $(1+\beta)$ moments (in $v$, uniformely in time and the $x$ variable) which implies that $\xi \to \hat u(t,k,\xi)$ is a bounded and $(1+\beta)$-Hölder continuous function, uniformely in $t,k$, and we have thanks to Appendix \ref{appendix:Fractional}.
\[\big|\hat \rho(t,k)\big| \le \int_0^t \sup_{k\in \RR^d}\|\hat u(s,k,\cdot)\|_{\cC^{1+\beta}_b}  \dd s. \]
 Note also that 
\[\|\hat u(t,k,\cdot)\|_{\cC^{1+\beta}_b} \lesssim \sup_{t\in[0,T]}\int_{\RR^{d}\times \RR^d}\big(1+|v|\big)^{1+\beta}u(t,\dd v,\dd x),\]
and thanks to Remark \ref{remark:1+beta} and Corollary \ref{corollary:pde1+beta} this is finite, as soon as the main Hypothesis is satisfied.

\emph{Control of the second and third coordinates of the solutions by the first one.}
Now, take two measures solutions $(u_1, \rho_1,C_1)$ and $(u_2,\rho_2,C_2)$ with the same initial conditions $u_0$ and $C_0$. For now, we allow our bound in the $\lesssim$ to depends on $\|C_0\|_{H^{m+2}}$, $\|K_C\|_{H^{\frac{3m}{2}}}$, $\rho_i(T,\RR^d)$, $\sup_{t\in[0,T]} u_i(t,\RR^d,\RR^d)$ and $\sup_{t\in[0,T]}\|C_i(t,\cdot)\|_{H^m}$.

Remark that thanks to the previous discussion, we have
\begin{equation}\label{eq:BoundHatRho}
\big\|(\hat{\rho}_1 - \hat{\rho}_2)(s,\cdot)\big\|_{\infty} \lesssim  \int_0^t \sup_{k}\|(\hat{u}_1 - \hat{u}_{2})(s,k,\cdot)\|_{\cC^{1+\beta}_b} \dd s.
\end{equation} 
Remark that $w=C_1-C_2$ satisfies the following equation :
\[\partial_t w  = \frac{\sigma_C^2}{2}\Delta w - (K_C*\rho_1) w + \big(K_C*(\rho_1 - \rho_2)\big) C_2, \quad w_0 = 0.\]
Thanks to Lemma \ref{lem:smoothness_C}, we have
\begin{multline*}
\big\|C_1(t,\cdot) - C_2(t,\cdot)\big\|_{H^{m+2}} \lesssim \big(1+\|K_C\|_{C^{m+2}_b}\big)^{m+2} \big(1+\rho_1(T,\RR^d)\Big)^{m+2} 
\\ 
\times \int_0^t\Big\| \big(K_C*(\rho_1 - \rho_2)\big)(s,\cdot) C_2(s,\cdot) \Big\|_{H^{m+2}} \dd s.
\end{multline*}
Furthermore 
\begin{align*}
\Big\| \big(K_C*(\rho_1 - \rho_2)\big)(s,\cdot) C_2(s,\cdot) \Big\|^2_{H^{m+2}}
=\int_{\RR^d} \big(1+|k|^2\big)^{m+2}\Big|\big(\hat{K}_C(\hat{\rho}_1 - \hat{\rho_2})\big)*\hat{C}_2(s,k)\Big|^2\dd k,
\end{align*}
and by using the fact that for all $k,k'$, $\big(1+|k|^2\big)^{m+2} \le \big(1+|k'|^2\big)^{m+2} \big(1+|k-k'|^2\big)^{m+2}$, by Young inequality,
\begin{align*}
\Big\| \big(K_C*(\rho_1 - \rho_2)\big)(s,\cdot) C_2(s,\cdot) \Big\|_{H^{m+2}} 
    \lesssim &
\|C_2(s,\cdot)\|_{H^{m+2}}  \int_{\RR^d} |\hat K_C(k)| \big(1+|k|^2\big)^{\frac{{m+2}}{2}} \dd k 
\\ &
\quad \times \big\|\hat{\rho}_1(s,\cdot)-\hat{\rho}_2(s,\cdot)\big\|_{\infty}\\
    \lesssim & \sup_{t\in[0,T]}\|C_2(t,\cdot)\|_{H^{m+2}} \|K_C\|_{H^{2m+2}}  \big\|(\hat{\rho}_1-\hat{\rho}_2)(s,\cdot)\big\|_{\infty}.
\end{align*}

Finally, we have the following bound for $C_1-C_2$,
\begin{equation} \label{eq:boundHmC1C2}
\big\|C_1(t,\cdot) - C_2(t,\cdot)\big\|_{H^m} \lesssim \int_{0}^{t} \int_{0}^s \sup_{k\in \RR^d} \big\|(\hat{u}_1 - \hat{u}_2)(s,k,\cdot)\Big\|_{\cC^{1+\beta}_b} \dd r \dd s.
\end{equation}

\emph{Gronwall type argument :}

Thanks to Equation \eqref{eq:MildFourier}, we have
\begin{align}
\hat u_1(t,k,\xi) &- \hat u_2(t,k,\xi)  = \nonumber \\
&\int_0^t  
        i\xi(t-s)\cdot\big( (\widehat{\nabla C_1}-\widehat{\nabla C_2})*\hat{u}_1\big)\big(s,k,\xi(t-s)\big)e^{-\frac{\sigma^2}{2}\int_0^{t-s} |\xi(r)|^2 \dd r} \dd s \label{eq:BoundHatU1}\\ &
        +\int_0^t  
        i\xi(t-s)\cdot\big(\widehat{\nabla C_2}*(\hat{u}_1-\hat{u}_2)\big)\big(s,k,\xi(t-s)\big)e^{-\frac{\sigma^2}{2}\int_0^{t-s} |\xi(r)|^2 \dd r
    }\dd s \label{eq:BoundHatU2}\\&
        + \int_0^t \hat{G}\big(\xi(t-s)\big)\big(\hat{u}_1(s,k,0) - \hat{u}_2(s,k,0) \big) e^{-\frac{\sigma^2}{2}\int_0^{t-s} |\xi(r)|^2 \dd r 
    } \dd s \label{eq:BoundHatU3}\\&
    + \int_0^t \hat{G}\big(\xi(t-s)\big)\hat{K}(k)\big( \hat{\rho}_1(s,k)-\hat{\rho}_2(s,k)\big) e^{-\frac{\sigma^2}{2}\int_0^{t-s} |\xi(r)|^2 \dd r 
    } \dd s \label{eq:BoundHatU4}\\&
      - \int_0^t \Big(\big( \hat{K} (\hat{\rho}_1-\hat{\rho}_2\big)*\hat{u}_1\Big)\big(s,k,\xi(t-s)\big)
      e^{-\frac{\sigma^2}{2}\int_0^{t-s} |\xi(r)|^2 \dd r 
    }\dd s
    \label{eq:BoundHatU5}\\&
      - \int_0^t \Big(\big( \hat{K} \hat{\rho}_2\big)*(\hat{u}_1-\hat{u}_2)\Big)\big(s,k,\xi(t-s)\big)
      e^{-\frac{\sigma^2}{2}\int_0^{t-s} |\xi(r)|^2 \dd r 
    }\dd s. \label{eq:BoundHatU6}
\end{align} 
Let us recall a basic inequality on Hölder norms : if $f_1$ and $f_2$ are two $(1+\beta)$-Hölder continuous functions from $\RR^d$ to $\RR^d$, then
\[\|f_1\cdot f_2\|_{\cC^{1+\beta}_b} \lesssim \|f_1\|_{\cC^{1+\beta}_b}\|f_2\|_{\cC^{1+\beta}_b}.\]

When dealing with $\eqref{eq:BoundHatU1}$, by taking $f_1(\xi) = \xi(t-s) e^{-\int_0^{t-s} \big|\xi(r)\big|^2\dd r}$, one has (since $D_\xi \xi(r) = e^{-\lambda r} Id$), 
\[Df_1(\xi) = e^{-\lambda (t-s)} e^{-\int_0^{t-s} \big|\xi(r)\big|^2\dd r} Id + \sigma\xi(t-s)\otimes \int_0^{t-s} \xi(r) \dd r e^{-\int_0^{t-s} \big|\xi(r)\big|^2\dd r}\]
and 
\begin{multline*}
D^2f_1(\xi) = \sigma^2 e^{-\lambda (t-s)} e^{-\int_0^{t-s} \big|\xi(r)\big|^2\dd r} \int_0^{t-s} \xi(r)\otimes Id \dd r  
\\
+  \sigma^2\int_0^{t-s} I_d \otimes \xi(r) \dd r e^{-\int_0^{t-s} \big|\xi(r)\big|^2\dd r} + \sigma^2\frac{1-e^{-\lambda(t-s)}}{\lambda} \xi(t-s) \otimes Id\, e^{-\int_0^{t-s} \big|\xi(r)\big|^2\dd r}\\
+
\sigma^4 \xi(t-s)\otimes \bigg(\int_0^{t-s} \xi(r) \dd r\bigg)^{\otimes 2}e^{-\int_0^{t-s} \big|\xi(r)\big|^2\dd r}.
\end{multline*}
Hence, thanks to the second inequality of Proposition \ref{prop:hypo} with $n=0,1,2$, we have
\[\|f\|_{\cC^{1+\beta}_b} \lesssim \frac{1}{\sqrt{t-s}}.\]
Take
\[f_2(\xi) = \big( (\widehat{\nabla C_1}-\widehat{\nabla C_2})*\hat{u}_1\big)\big(s,k,\xi(t-s)\big),\]
by using the previous bound for $C^1-C^2$, we have
\begin{multline*}
\|f_2\|_{\cC^{1+\beta}_b} \lesssim \sup_{k}\|\hat{u}_1(s,k,\cdot)\|_{\cC^{1+\beta}_b}\int_0^s \int_0^r \sup_{k}\|\hat{u}_1(\tau,k,\cdot)-\hat{u}_2(\tau,k,\cdot)\|_{\cC^{1+\beta}_b} \dd \tau \dd r\\
\lesssim \int_0^s \int_0^r \sup_{k}\|\hat{u}_1(\tau,k,\cdot)-\hat{u}_2(\tau,k,\cdot)\|_{\cC^{1+\beta}_b} \dd \tau \dd r
\end{multline*}
Hence, we have the following bound for \eqref{eq:BoundHatU1} :
\begin{equation*}
\sup_{k}\big\|\eqref{eq:BoundHatU1}\big\|_{\cC^{1+\beta}_b} \lesssim \int_{0}^t \frac{1}{\sqrt{t-s}} \int_0^s \int_0^r \sup_{k}\|(\hat{u}_1-\hat{u}_2)(\tau,k,\cdot)\|_{\cC^{1+\beta}_b} \dd \tau\dd r\dd s.
\end{equation*}
The proof for \eqref{eq:BoundHatU2} when taking $f_1$ to be the same and $f_2 = \big(\widehat{\nabla C_2}*(\hat{u}_1-\hat{u}_2)\big)\big(s,k,\xi(t-s)\big)$, since $\|f_2\|_{\cC^{1+\beta}_b} \lesssim \sup_{k}\big\| \hat{u}_1(s,k,\cdot) - \hat{u}_2(s,k,\cdot) \big\|_{\cC^{1+\beta}_b}$, and 
\begin{equation*}
\sup_{k}\big\|\eqref{eq:BoundHatU2}\big\|_{\cC^{1+\beta}_b} \lesssim \int_{0}^t \frac{1}{\sqrt{t-s}}  \sup_{k}\big\|(\hat{u}_1-\hat{u}_2)(s,k,\cdot)\big\|_{\cC^{1+\beta}_b} \dd s.
\end{equation*}

In \eqref{eq:BoundHatU3} and \eqref{eq:BoundHatU4}, since $G$ has finite $(1+\beta)$ moments, $\hat G$ is $(1+\beta)$-Hölder continuous. Take $f_1(\xi)= \widehat{G}\big(\xi(t-s)\big) e^{-\frac{\sigma^2}{2}\int_0^{t-s}\big|\xi(r)\big|^2 \dd r}$, we have $\|f_1\|_{\cC^{1+\beta}_b}\lesssim 1$. Furthermore take $f_2 = \hat{u}_1(s,k,0) - \hat{u}_2(s,k,0) + \hat{K}(\xi)\big(\hat{\rho}_1(s,k) - \hat{\rho}_2(s,k)\big)$. Thanks to \eqref{eq:BoundHatRho} we have $\|f_2\|_{\cC^{1+\beta}_b} \lesssim \sup_{k}\|\hat{u}_1(s,k,\cdot)\|_{\cC^{1+\beta}_b} + \int_0^s \sup_{k}\|\hat{u}_1(s,k,\cdot)\|_{\cC^{1+\beta}_b} \dd r$, and we have
\begin{multline*}
\sup_{k\in \RR^d} \|\eqref{eq:BoundHatU3} + \eqref{eq:BoundHatU4}\|_{\cC^{1+\beta}_b} \lesssim \int_0^t \sup_{k}\|\hat{u}_1(s,k,\cdot) - \hat{u}_2(s,k,\cdot) \|_{\cC^{1+\beta}_b} \dd s  \\ + \int_0^t \int_0^s \sup_{k}\|\hat{u}_1(s,k,\cdot) - \hat{u}_2(s,k,\cdot)\|_{\cC^{1+\beta}_b} \dd r \dd s.
\end{multline*}
In \eqref{eq:BoundHatU5}, take $f_1 = e^{-\frac{\sigma^2}{2} \int_0^{t-s} |\xi(r)|^2\dd r}$ and $f_2 = \Big(\big( \hat{K} (\hat{\rho}_1-\hat{\rho}_2\big)*\hat{u}_1\Big)\big(s,k,\xi(t-s)\big)$.Remind that the convolution is in the $k$ variable, hence thanks to Young inequality,
\begin{align*}
\|f_2\|_{\cC^{1+\beta}_b}    \lesssim & \| \hat{K}_m (\hat{\rho}_1-\hat{\rho}_2)(s,\cdot)\big\|_{L^1} \sup_{k\in\RR^d} \big\|\hat{u}_1(s,k,\cdot)\big\|_{\cC^{1+\beta}_b}\\
\lesssim & \|K\|_{H^{m}} \int_{0}^s \sup_{k} \|\hat{u}_1(r,k,\cdot) - \hat{u}_2(r,k,\cdot)\|_{\cC^{1+\beta}_b}\dd r.
\end{align*}
The same holds for \eqref{eq:BoundHatU6} with $f_2 = \Big(\big( \hat{K} \hat{\rho}_2\big)*(\hat{u}_1-\hat{u}_2)\Big)$, and we have
\[\|f_2\|_{\cC^{1+\beta}_b} \lesssim \sup_{k} \|\hat{u}_1(r,k,\cdot) - \hat{u}_2(r,k,\cdot)\|_{\cC^{1+\beta}_b}.\]
Hence, one have
\begin{multline*}
\sup_{k\in \RR^d} \|\eqref{eq:BoundHatU5} + \eqref{eq:BoundHatU6}\|_{\cC^{1+\beta}_b} \lesssim \int_0^t \sup_{k}\|\hat{u}_1(s,k,\cdot) - \hat{u}_2(s,k,\cdot) \|_{\cC^{1+\beta}_b} \dd s  \\ + \int_0^t \int_0^s \sup_{k}\|\hat{u}_1(r,k,\cdot) - \hat{u}_2(r,k,\cdot)\|_{\cC^{1+\beta}_b} \dd r \dd s,
\end{multline*}
and finally
\begin{multline*}
\sup_{k\in \RR^d}\big\|\hat{u}_1(t,k,\cdot)-\hat{u}_2(t,k,\cdot)\big\|_{\cC^{1+\beta}_b} \lesssim \\
\int_0^t\Big(1+\frac{1}{\sqrt{t-s}}\Big) \sup_{k}\|\hat{u}_1(s,k,\cdot) - \hat{u}_2(s,k,\cdot) \|_{\cC^{1+\beta}_b} \dd s 
 \\ 
 + \int_0^t \int_0^s \sup_{k}\|\hat{u}_1(r,k,\cdot) - \hat{u}_2(r,k,\cdot)\|_{\cC^{1+\beta}_b} \dd r \dd s\\
+ \int_0^t \frac{1}{\sqrt{t-s}} \int_0^s \int_0^r \sup_{k}\|\hat{u}_1(\tau,k,\cdot) - \hat{u}_2(\tau,k,\cdot)\|_{\cC^{1+\beta}_b} \dd \tau \dd r \dd s.
\end{multline*}
We can conclude by using the Gronwall type lemma of Appendix \ref{appendix:gronwall}, with $A_0=0$, and we have $\hat{u}_1 = \hat{u}_2$ for all $t,k,\xi$. using the bound for $C_1-C_2$ and $\hat{\rho}_1
-\hat{\rho}_2$, we can conclude that $(u_1,\rho_1,C_1)= (u_2,\rho_2,C_2)$, which ends the proof.\end{proof}

\begin{proof}[Proof of Theorem  \ref{theorem:main}]
From the tightness of $\{Q^{N}\}_{N\in\NN}$ and Theorem \ref{theorem:limitsupport} we obtain the convergence of subsequences. By Theorem \ref{theorem:uniqueness} we obtain the convergence of the full sequence and thus the desired result.
\end{proof}

\section{Smoothness of the solution}\label{sec:smoothness}

We end by recalling Hypothesis \ref{hypo:smoothness} and proving 	a theorem for the smoothness of the solutions. 
\begin{itemize}
\item For all $N\geq 0$, 
$\sup_{\xi\in\RR^d} |\hat{G}(\xi)\big|\big(1+|\xi|^2)^N < \infty$.
\item For all $N\geq 0$, 
$\sup_{\xi\in\RR^d} |\hat{K}(k)\big|\big(1+|\xi|^2)^N < \infty$.
\item
For all $m \geq 0$, $C_0,K_C \in H^m$ and $K_C \in C^m_b$.
\end{itemize}

\begin{proof}[Proof of Theorem \ref{theorem:regularity}]
Following Desvillette and Villani \cite{desvillettes2001}, one only has to prove that if there exists $n\ge 0$ such that 
\begin{equation}\label{eq:induction}
\sup_{t\in[0,T]} \sup_{k,\xi \in \RR^d} \big(1+|k|^2+|\xi|^2\big)^{\frac{n}{6}}\big|\hat{u}(t,k,\xi)\big| < +\infty,
\end{equation}
then for all $t_0>0$,
\[\sup_{t\in[t_0,T]} \sup_{k,\xi \in \RR^d} \big(1+|k|^2+|\xi|^2\big)^{\frac{n+1}{6}}\big|\hat{u}(t,k,\xi)\big| < +\infty,\]
and then conclude by induction.
Note first that 
\[|k|^2 + |\xi(r)|^2 = 
\left( 
\begin{pmatrix}
e^{-2\lambda r} & e^{-\lambda r}\frac{1-e^{-\lambda r}}{ \lambda} \\  e^{-\lambda r}\frac{1-e^{-\lambda r}}{ \lambda} & 1+ \Big(\frac{1-e^{-\lambda r}}{ \lambda}\Big)^2
\end{pmatrix}
\begin{pmatrix}
\xi \\ k
\end{pmatrix} 
\right) 
\cdot \begin{pmatrix}
\xi \\ k
\end{pmatrix}.\]
Furthermore, the determinant of the previous non negative matrix is equal to $e^{-2\lambda r} \ge e^{-2\lambda T}$, and (thanks to Cauchy-Schwarz inequality) its largest eigenvalue is at most $1+T^2$.  Hence its smallest eigenvalue is greater than $\frac{e^{-2\lambda T}}{1+T^2}$ and there exists $c>0$ such that for all $r\in[0,T]$,
\[\frac1c \big(|k|^2 + |\xi|^2\big) \le |k|^2 + |\xi(r)|^2 \le c \big(|\xi|^2+ |k|^2).\]
Since \eqref{eq:induction} is true for $t=0$, one have
\[\Big|\hat{u}_0\big(t,k,\xi(t)\big)\Big| \lesssim \Big(1+|k|^2+\big|\xi(r)\big|^2\Big)^{-\frac{n}{6}} \lesssim \big(1+|k|^2 + |\xi|^2\big)^{-\frac{n}{6}}.\]
Furthermore, thanks to the hypoelliptic estimates Proposition \ref{prop:hypo}, we have
\begin{equation}\label{eq:boundhypo}\exp\left(-\frac{\sigma^2}{2}\int_{0}^t |\xi(r)|^2 \dd r\right) \lesssim \frac{1}{t^{\frac23}}\Big(1+|k|^2+\big|\xi\big|^2\Big)^{-\frac{1}{6}}
\end{equation}
and this close the bound for the first term of \eqref{eq:MildFourier}. This also allows us to bound the second last two terms of \eqref{eq:MildFourier}. Indeed, since $G$ and $K$ are decaying faster than any polynomials, one has
 \[\Big|G\big(\xi(t-s)\big)\Big| \big|\hat{u}(s,k,0)\big| \lesssim  \big(1+|\xi(t-s)|^2\big)^{-\frac{n}{6}}\big(1+|k|^2\big)^{-\frac{n}{6}} \lesssim \big(1+|k|^2+ |\xi(t-s)|^2\big)^{-\frac{n}{6}}\]
 and
\begin{multline*}
\Big|G\big(\xi(t-s)\big)\Big| \big|\hat{K}(k) \hat{\rho}(t,k)\big| \lesssim_{\rho(T,\RR^d)}  \big(1+|\xi(t-s)|^2\big)^{-\frac{n}{6}}\big(1+|k|^2\big)^{-\frac{n}{6}} 
\\
\lesssim \big(1+|k|^2+ |\xi(t-s)|^2\big)^{-\frac{n}{6}} \lesssim \big(1+|k|^2+ |\xi|^2\big)^{-\frac{n}{6}}. \end{multline*}
Thanks to the bound
\[\big(1+|k|^2 + |\xi|^2\big) \lesssim \big(1+|k-k'|^2 + |\xi|^2\big) \big(1+|k'|^2 + |\xi(t-s)|^2\big),\]
we also have the following bound,
\begin{align*}
\Big|\big(\hat{K} \hat{\rho} \big)* \hat{u}(s,\xi(t-s)\big) \Big|\big(1+|k|^2 + |\xi|^2\big)^{\frac{n}{6}} \lesssim & \int \big|\hat K(k)\big| \big|\hat{\rho}(s,k)\big| \big(1+|k|^2  \big)^{\frac{n}6} \dd k \\
\lesssim & 1.
\end{align*}
Hence using \eqref{eq:boundhypo}, one can bound the last to lines of \eqref{eq:MildFourier} by a constant times
\[ \big(1+|\xi|^2 + |k|^2\big)^{-\frac{n+1}{6}} \int_0^t (t-s)^{-\frac23} \dd s.\]
Finally $C_0$ and $K_C$ are regular enough, $C\in H^m$ for all $m\in \NN$. One has
\begin{align*}
\Big|\widehat{\nabla C}*\hat{u}(s,k,\xi)\big| \big(1+|k|^2 + |\xi|^2\big)^{\frac{n}6} \lesssim & \int_{\RR^d} \big|\widehat{\nabla C}(t,k)\big| \big(1+|k|^2\big) \dd k \\
\lesssim & \|C(t,\cdot)\|_{H^{2+m+\frac{n}3}}
\end{align*}
for some $m>\frac{d}2$, and one can conlcude by using Lemma \ref{lem:smoothness_C}. Finally, since $\xi(r) \le |\xi| + r|k|$ One can bound the last line of \eqref{eq:MildFourier} by
\begin{multline*}
    \big(1+|k|^2 + |\xi|^2\big)^{-\frac{n}{6}}\int_0^t \big(|\xi| + r|k| \big) e^{-\frac{\sigma^2}{2}\int_0^s |\xi(r)| \dd r}\dd s \\ \lesssim \big(1+|k|^2 + |\xi|^2\big)^{-\frac{n}{6}} \int_0^t \big(|\xi| + r|k| \big) e^{-c s^3 |k|^2 +s |\xi|^2 }\dd s.
\end{multline*}
Following Desvillette and Villani \cite{desvillettes2001} Lemma 5.3, one can conclude that the latter is no greater than a constant times 
\[\big(1+|k|^2 + |\xi|^2\big)^{-\frac{n}{6}},\]
which ends the proof.
\end{proof}

\appendix

\section{Reminder on curvilinear abscissa}\label{appendix:curvinlinear}

Let us remind that if $X$ is a $C^1$ curve in $\RR^d$, parametrized by $t\in[0,T]$
\[s(t) = \int_0^t \big|X'(r)\big| \dd r,\]
is the curvilinear abscissa of $X$.
Hence, let us define $\tilde X$ such that
\[\tilde X \big(s(t)\big) = X(t),\]
where $\tilde X$ is parametrized by $t\in \big[0,s(T)\big]$. Hence, we have
\[s'(t)\tilde X'\big(s(t)\big) = X'(t),\]
and finally for all $t\in[0,s(T)]$,
\[|\tilde X'(t)| = 1,\]
and  the velocity of $\tilde X$ is of norm $1$. Finally, it is possible to parameterize the following spatial set (independently of the speed of the curve) :
\[\XX_t = \{X(s), s\in[0,t]\} = \{\tilde X\big(s(r)\big), r\in[0,t]\} = \{\tilde X(r) , r\in[0,s(t)]\}.\]
Hence, 
\begin{align*}
\big \langle f , \delta_{\XX_t} \big \rangle =& \int_{0}^{s(t)} f\big( \tilde X(r) \big) \dd r\\
=&\int_{0}^{s(t)} f\Big(X\big(s^{-1}(r)\big)\Big)\dd r \\
=&
\int_{0}^{t} f\big(X(r)\big) \big|X'(r)\big|\dd r \\
=&
\left \langle f , \int_{0}^{t} \big|X'(r)\big| \delta_{X(r)} \dd r \right \rangle.
\end{align*}

Which gives 
\begin{equation}\label{eq:curvilinear}
\delta_{\XX_t}(\dd x) = \int_{0}^t \big|X'(r) \big| \delta_{X(r)}(\dd x) \dd r
\end{equation}

\section{A generalized Gronwall Lemma}\label{appendix:gronwall}
\begin{lemma}
Let $n\ge 1$. Let $A_0,A_1,a_1,\cdots,A_n,a_n\in \RR_+$, $q_1,\cdots,q_n >1$ and let $f$ be a positive measurable function such that for all $t\in[0,T]$,
\begin{multline*}
f(t) \le    A_0 + A_1\int_0^t \big(1+a_1(t-s_1)^{-\frac{1}{q_1}}\big)f(s_1) \dd s_1 + \cdots 
\\+ A_n 
    \int_0^t 
        \Big(1+a_n(t-s_1)^{-\frac{1}{q_n}}\Big) 
        \int_0^{s_1} \cdots \int_0^{s_{n-1}}f(s_n) \dd s_n \cdots \dd s_1.
        \end{multline*}
There exists a constant $C>0$ which may depend on all the parameters such that  for all $t\in[0,T]$,
\[f(t)  \lesssim A_0 e^{Ct}.\]
\end{lemma}

\begin{proof}
Let $q < \min_{1\le i\le n}{q_i}$ and let $p>1$ such that $\frac{1}{p}+\frac{1}{q} = 1$. By using Hölder inequality, and Jensen Inequality, we have, for all $1\le k\le n$
\begin{multline*}
\int_0^t 
        \Big(1+a_k(t-s_1)^{-\frac{1}{q_n}}\Big) 
        \int_0^{s_1} \cdots \int_0^{s_{k-1}}f(s_n) \dd s_k \cdots \dd s_1
        \\
         \lesssim  \bigg( \int_{0}^t \int_0^{s_1} \cdots \int_0^{s_{k-1}} f(s_k)^p \dd s_k \cdots \dd s_2\dd s_1 \bigg)^{\frac1p} 
        \end{multline*}

There exists some constants $B>0$ depending on $T$, $A_1,\cdots,A_n$ and $a_1,\cdots,a_n$, $q,q_1,\cdots,q_n$ and $n$ such that 
\[
f(t) \le    A_0 + B\bigg(\int_0^t g_0(s) + g_1(s)^p + \cdots g_{n-1}(s)^p \dd s\bigg)^{\frac1p},
        \]
where $g_0(t) = f(t)^p$ et $g'_k(t) = g_{k-1}(t)$. Finally, since $g_1(t) + \cdots + g_{n-1}(t) = \int_0^t g_1(s) + \cdots + g_{n-2}(s)\dd s$, 
there exists a constant $c$ and a constant $C>0$ such that 
\[g_0(t) + \cdots +g_{n-1}(t) \le c A_0^p  + p C \int_{0}^t g_0(s) + \cdots +g_{n-1}(s) \dd s.\]
We conclude by the classical Hölder inequality, and we have
\[f(t) \lesssim A_0e^{ C t}.\]
\end{proof}

\section{Toolbox on fractional Laplacian}\label{appendix:Fractional}

Let us recall that we define the non local operator $(-\Delta)^{\frac12}$ for sufficient regular functions $f$ by the formula
\[(-\Delta)^{\frac12}f(x) = \frac{\Gamma\Big(\frac{d+1}{2}\big)}{\pi^{\frac{d+1}{2}}} V.P. \int_{\RR^d} \frac{f(x) - f(y)}{|x-y|^{d+1}} \dd y,\] 
where $V.P.$ denotes the principal value. We then have the following lemma :
\begin{lemma}\label{lemma:FractionalLaplacian}
Let $\beta>0$. The operator $(-\Delta)^{\frac12}$ is well-defined on the space of Hölder continuous functions $\cC^{1+\beta}_b(\RR^d;\RR)$ to the space of bounded functions, and for $f,g \in \cC^{1+\beta}_b(\RR^d;\RR)$ we have
\[\|(-\Delta)^{\frac12}f - (-\Delta)^{\frac12}g\|_{\infty, \RR^d} \lesssim \|f-g\|_{\cC^{1+\beta}_b}.\]
\end{lemma}

\begin{proof}
First, let us remark that 
\[\bigg|\int_{\RR^d\setminus B(x,1)} \frac{f(x)-f(y)}{|x-y|^{d+1}} \dd y \bigg| \lesssim \|f\|_{\infty}\]
since $f$ is a bounded function. Furthermore, 
for all $1>\eps>0$, we have
\[\int_{B(x,1)\setminus B(x,\eps)} Df(x)\frac{y-x}{|x-y|^{d+1}} \dd y= 0.\]
Finally, remark that
\begin{multline*}
\int_{B(x,1)\setminus B(x,\eps)} \frac{f(x) - f(y)}{|x-y|^{d+1}} \dd y = \int_{B(x,1)\setminus B(x,\eps)} \int_0^1 Df(l(y-x)+x)\dd l\frac{x-y}{|x-y|^{d+1}} \dd y\\
= \int_{B(x,1)\setminus B(x,\eps)} \int_0^1 \big(Df(l(y-x)+x)-Df(x)\big)\dd l\frac{x-y}{|x-y|^{d+1}} \dd y.
\end{multline*}
Hence, one can use the fact the $Df$ is $\beta$-Hölder continuous, and one have
\[\bigg|\int_{B(x,1)\setminus B(x,\eps)} \frac{f(x) - f(y)}{|x-y|^{d+1}} \dd y\bigg| \lesssim \|f\|_{\cC^{1+\beta}_b}.\]
By using the dominated convergence theorem, $(-\Delta)^{\frac12}f$ is well-defined, and we have the wanted bound, since the operator is linear. 
\end{proof}

\section{Duhamel formulation of kinetic transport equation}\label{appendix:Mild}  Let $f\in L^\infty\big([0,T]\times \RR^d\times \RR^d\big)$,  us look at the equations for the characteristics lines starting from $\tilde{\xi},\tilde{k}\in \RR^d$ of the first order  equation
\[\partial_t h(t,k,\xi) - k\cdot\nabla_\xi h(t,k,\xi) + \lambda \xi\cdot\nabla_\xi h(t,k,\xi) = -\frac{|\xi|^2 \sigma^2}{2}h(t,k,\xi)+ f(t,k,\xi).\]
We have
\[\begin{cases}
\tilde \xi'(t) &= - \tilde k'(t) + \lambda \tilde \xi'(t) \\ \tilde k'(t) &= 0
\end{cases}\]
Hence, $\tilde \xi(t) =  \Big(\tilde{\xi}-\frac{\tilde k}{\lambda}\Big) e^{\lambda t} + \frac{\tilde k}{\lambda} $ and $\tilde k(t) = \tilde k$
Finally, one have
\[\partial_t h\big(t,\tilde k, \tilde \xi(t)\big)=-\frac{|\tilde \xi(t)|^2 \sigma^2}{2}h\big(t,\tilde k,\tilde \xi(t)\big)+ f\big(t,\tilde k,\tilde \xi(t)\big) \]
and by solving this ordinary differential equation, one have
\[h\big(t,\tilde k,\tilde \xi(t)\big) = h_0(\tilde k,\tilde \xi)e^{-\frac{\sigma^2}{2}\int_0^t |\tilde \xi(r)|^2 \dd r} + \int_0^t f\big(s,\tilde k,\tilde \xi(s)\big)e^{-\frac{\sigma^2}{2}\int_s^t |\tilde \xi(r)|^2 \dd r}\dd s.\]
Now, let us fix $t\in[0,T]$ and let us  take $\tilde k = k$ and $\tilde \xi = \Big(\xi-\frac{ k}{\lambda}\Big) e^{-\lambda t} + \frac{ k}{\lambda}$, such that $\tilde \xi(t) = \xi$ and we have
\begin{multline*}
h(t,k,\xi) = h_0\bigg(k,\Big(\xi-\frac{ k}{\lambda}\Big) e^{-\lambda t} + \frac{ k}{\lambda}\bigg)\exp\bigg(-\frac{\sigma^2}{2} \int_0^t \Big|\Big(\xi-\frac{ k}{\lambda}\Big) e^{-\lambda r} + \frac{ k}{\lambda}\Big|^2 \dd r\bigg)\\
+\int_0^t f\bigg(s,k,\Big(\xi-\frac{ k}{\lambda}\Big) e^{-\lambda (t-s)} + \frac{ k}{\lambda}\bigg) \exp\bigg(-\frac{\sigma^2}{2} \int_0^{t-s} \Big|\Big(\xi-\frac{ k}{\lambda}\Big) e^{-\lambda r} + \frac{ k}{\lambda}\Big|^2 \dd r\bigg) \dd s\\
= h_0\big(k,\xi(t)\big) e^{-\frac{\sigma^2}{2}\int_0^t |\xi(r)|^2 \dd r} + \int_0^t f\big(s,k,\xi(t-s)\big) e^{-\frac{\sigma^2}{2} \int_0^{t-s} |\xi(r)|^2 \dd r}
\end{multline*}
with  $\xi(r) = \left(\xi - \frac{k}{\lambda}\right)e^{-\lambda r} + \frac{k}{\lambda}$.


\end{document}

%% file: 00lettres_latex.tex
\newcommand{\cC}{\mathcal{C}}

\newcommand{\cF}{\mathcal{F}}

\newcommand{\cM}{\mathcal{M}}

\newcommand{\EE}{\mathbbm{E}}

\newcommand{\NN}{\mathbbm{N}}

\newcommand{\PP}{\mathbbm{P}}

\newcommand{\RR}{\mathbbm{R}}

\newcommand{\XX}{\mathbbm{X}}



%% file: 00arbres.tex
\usetikzlibrary{shapes.symbols}
\tikzset{
	xi/.style={circle,fill=blue!10,draw=black,inner sep=0pt,minimum size=1.2mm},
	xib/.style={circle,fill=blue!10,draw=black,inner sep=0pt,minimum size=1.6mm},
	not/.style={circle,fill=black,draw=black,inner sep=0pt,minimum size=0.5mm},
	>=stealth,
	}
\makeatletter
\def\DeclareSymbol#1#2#3{\expandafter\gdef\csname MH@symb@#1\endcsname{\tikz[baseline=#2,scale=0.15]{#3}}}
\def\<#1>{\csname MH@symb@#1\endcsname}

\makeatother

\DeclareSymbol{tree1}{0.5}{\draw (-1,1) node[not] {};}
\DeclareSymbol{tree12}{0.5}{\draw (-0.5,1.5) node[not] {} -- (0,0.5); \draw (0.5,1.5) node[not] {} -- (0,0.5);}
\DeclareSymbol{tree113}{0.5}{\draw (0,0) -- (0,.5) ; \draw (-0.7,1.5) node[not] {} -- (0,.5); \draw (0,1.5) node[not] {} -- (0,.5); \draw (.7,1.5) node[not] {}  -- (0,.5);}
\DeclareSymbol{tree123}{0.5}{\draw (1.4,1) node[not] {} -- (.7,0) ; \draw (0.7,0)  -- (0,1); \draw (.7,2) node[not] {} -- (0,1); \draw (0,2) node[not] {}  -- (0,1); \draw (-0.7,2) node[not] {}  -- (0,1);}
\DeclareSymbol{tree133}{0.5}{\draw (1.4,1) node[not] {} -- (.7,0) ;\draw (.7,1) node[not] {} -- (.7,0) ; \draw (0.7,0)  -- (0,1); \draw (.7,2) node[not] {} -- (0,1); \draw (0,2) node[not] {}  -- (0,1); \draw (-0.7,2) node[not] {}  -- (0,1);}
\DeclareSymbol{tree132}{0.5}{\draw (1.4,1) node[not] {} -- (.7,0) ;\draw (.7,1) node[not] {} -- (.7,0) ; \draw (0.7,0)  -- (0,1); \draw (.7,2) node[not] {} -- (0,1); \draw (-0.7,2) node[not] {}  -- (0,1);}


%% file: 00theorem.tex
\theoremstyle{plain}

\newtheorem{theorem}{Theorem}[section]
\newtheorem*{theorem*}{Theorem} 
\newtheorem{corollary}[theorem]{Corollary}
\newtheorem{hypothesis}[theorem]{Hypothesis}
\newtheorem{lemma}[theorem]{Lemma}
\newtheorem{proposition}[theorem]{Proposition}

\theoremstyle{definition}

\newtheorem{remark}[theorem]{Remark}
\newtheorem{definition}[theorem]{Definition}